\documentclass{article}
\usepackage[T1]{fontenc}
\usepackage{amssymb, amsmath, amsthm, graphicx, subfigure}
\usepackage[dvipsnames]{xcolor}
\usepackage[colorlinks=true,linktoc=page]{hyperref}
\usepackage{braket}
\usepackage{esvect}
\usepackage[margin=1in]{geometry}
\usepackage{mathtools}
\mathtoolsset{showonlyrefs=true}

\theoremstyle{plain}
\newtheorem{theorem}{Theorem}[section]
\newtheorem{lemma}[theorem]{Lemma}
\newtheorem{corollary}[theorem]{Corollary}
\newtheorem{proposition}[theorem]{Proposition}

\theoremstyle{definition}
\newtheorem{definition}[theorem]{Definition}

\theoremstyle{remark}

\numberwithin{equation}{section}

\newcommand\numberthis{\addtocounter{equation}{1}\tag{\theequation}}

\newcommand{\e}{\mathrm{e}}

\newcommand{\alg}{\mathsf{ALG}}

\newcommand{\bis}{\mathsf{bis}}
\newcommand{\SK}{\mathrm{SK}}

\newcommand{\bg}{\mathbf{g}}

\newcommand{\bA}{\mathbf{A}}
\newcommand{\bB}{\mathbf{B}}

\newcommand{\bJ}{\mathbf{J}}

\newcommand{\bX}{\mathbf{X}}

\newcommand{\same}{\mathsf{same}}

\newcommand{\E}{\mathop{{}\mathbf{E}}}

\newcommand{\T}{\mathrm{T}}

\newcommand{\Po}{\mathrm{Po}}

\hypersetup{colorlinks,
    linkcolor=blue,
    citecolor=ForestGreen,  
    urlcolor=Mahogany,
}

\title{Hardness of sampling for the anti-ferromagnetic Ising model on random graphs}

\author{Neng Huang\thanks{University of Michigan. Work was primarily done when NH was affiliated with the University of Chicago, supported partly by the Institute for Data, Econometrics, Algorithms, and Learning (IDEAL) with NSF grant ECCS-2216912. Email: \texttt{nengh@umich.edu}} \and Will Perkins\thanks{Georgia Tech. Supported in part by NSF grant CCF-2309708. Email: \texttt{wperkins3@gatech.edu}} \and Aaron Potechin\thanks{University of Chicago. Email: \texttt{potechin@uchicago.edu}}}

\begin{document}
\maketitle

\begin{abstract}
    We prove a hardness of sampling result for the anti-ferromagnetic Ising model on random graphs of average degree $d$ for large constant $d$, proving that when the normalized inverse temperature satisfies $\beta>1$ (asymptotically corresponding to the condensation threshold), then w.h.p. over the random graph there is no stable sampling algorithm that can output a sample close in $W_2$ distance to the Gibbs measure.  The results also apply to a fixed-magnetization version of the model, showing that there are no stable sampling algorithms for  low but positive temperature max and min bisection distributions.  These results show a gap in the tractability of search and sampling problems:  while there are efficient algorithms to find near optimizers, stable sampling  algorithms cannot access the Gibbs distribution concentrated on such solutions.

    Our techniques involve extensions of the interpolation technique relating behavior of the mean field Sherrington-Kirkpatrick model to behavior of Ising models on random  graphs of average degree $d$ for large $d$.  While previous interpolation arguments compared the free energies of the two models, our argument compares the average energies and average overlaps in the two models.  

\end{abstract}

\section{Introduction}
\label{secIntro}

The study of disordered systems is at the intersection of statistical physics, probability theory, and computer science.  Models of disordered systems, such as the Edwards--Anderson model, Sherrington--Kirkpatrick model, random constraint satisfaction problems, and combinatorial optimization problems on random networks, have been studied from many different perspectives, including as toy models of complex physical systems and as a source of random computational problems on which to test algorithms and understand the sources of algorithmic intractability.  

In the algorithmic context, disordered systems present at least two natural algorithmic challenges: the first is the \textit{search problem}, that of finding a high quality solution to an optimization problem or finding a near ground state  in the language of statistical physics; the second is the \textit{sampling problem}, that of (approximately) sampling from a Gibbs measure that weights solutions exponentially by their objective value. 

One striking phenomenon that has been discovered is a gap between the tractability of search and sampling.  That is, for certain choices of parameters in some models of disordered systems, the search problem is tractable (there exist polynomial-time or even near-linear time algorithms to  find high quality solutions), while no efficient sampling  algorithms are known, and in fact broad classes of sampling algorithms can be proved to be ineffective.  So while high quality solutions can be found efficiently, these solutions are not typical, in the sense that the distribution of solutions found by a given search algorithm is far from the equilibrium Gibbs measure.

A prime example of this phenomenon occurs in the Sherrington--Kirkpatrick (SK) model~\cite{sherrington1975solvable}, an Ising model on the complete graph on $n$ vertices in which the $\binom{n}{2}$ coupling constants are independent Gaussian random variables (see e.g~\cite{panchenko_sherrington-kirkpatrick_2013} for a mathematical survey of the model).  Here, after appropriate normalization, the key parameter is $\beta$, the inverse temperature of the model.   Parisi famously predicted a formula for the limiting free energy of the SK model model as a function of $\beta$~\cite{parisi1979infinite,parisi1980sequence}; this formula is non-analytic at $\beta=1$, indicating the existence of the phase transition at this point.  The Parisi formula was proved rigorously by Guerra~\cite{guerra2003broken} and Talagrand~\cite{talagrand_parisi_2006}, with methods that have proved very influential in mathematics, physics, and computer science in the past two decades.   Along with the Parisi  formula comes information about the SK Gibbs measure itself; in particular, Parisi identified the `average overlap' of two independent samples from the Gibbs measure as an order parameter for the phase transition~\cite{parisi1983order} and developed the theory of a hierarchy of `replica symmetry breaking'.  Thus the equilibrium picture of the SK model is fairly well understood, with the phase transition identified at $\beta =1$. 

Turning to the algorithmic search  problem, the task is to find a solution $\sigma \in \{ \pm 1\}^n$ that (approximately) maximizes the probability mass function of the SK model; that is, the task of finding a ground state or near ground state of the model.   Montanari~\cite{montanari2021optimization}, partly inspired by the  work of Subag on spherical spin glasses~\cite{subag2021following}, proposed an algorithm for this task  based on approximate message passing, which finds an approximate optimizer under the widely believed assumption that the SK model has ``no overlap gap'' (sometimes also referred to as ``full replica symmetry breaking'') for large enough $\beta$. This approach was later extended to mixed $p$-spin models~\cite{alaoui_optimization_2021}, which can be thought of as the generalization of the Ising model to hypergraphs.

For the sampling problem, however, the picture is less optimistic. It was predicted by physicists that the simple Glauber dynamics should converge fast as long as $\beta < 1$~\cite{PhysRevLett.47.359}, but progress on a rigorous proof of this statement has only come recently. A series of recent works~\cite{bauerschmidt2019very,eldan2022spectral,chen2022localization,anari2021entropic} showed that Glauber dynamics does indeed converge quickly if $\beta < 1/4$ (more specifically, in $O(n \log n)$ steps). Using a different approach based on \emph{stochastic localization},  El Alaoui, Montanari, and Sellke~\cite{alaoui_sampling_2022} gave a different sampling algorithm that approximately samples from the SK Gibbs measure for $\beta < 1/2$, with the approximation in terms of the Wasserstein distance instead of the total variation distance guarantees given by previous works. This result was later improved to the entire replica-symmetric phase $\beta < 1$ by Celentano~\cite{10.1214/23-AOP1675}. On the other hand, obstacles seem to arise after the phase transition at $\beta = 1$. El Alaoui, Montanari, and Sellke showed that when $\beta > 1$, the onset of \emph{disorder chaos} naturally obstructs \emph{stable} sampling algorithms~\cite{alaoui_sampling_2022}. Similar ideas were used recently by El Alaoui and Gamarnik to rule out stable sampling algorithms for the symmetric binary perceptron problem~\cite{alaoui_hardness_2024}.

One can ask whether this gap in tractability for the search and sampling problems is universal for certain classes of disordered systems.  The SK model is an example of a \textit{mean field model}: all spins (variables) interact with each other.  Mean field models have the advantage of being mathematically tractable to some degree, but the disadvantage of being unrealistic from the physics perspective as well as the perspective of large real-world networks.  Seeking a trade-off between these two aspects, physicists have studied \textit{diluted mean-field models}: that is, statistical physics models on random graphs or hypergraphs of constant average degree~\cite{mezard2001bethe}.  These models inherit some of the symmetries of mean-field models while also having some non-trivial local geometry.  The study of diluted mean field models in physics led to rich and surprising predictions for the behavior of optimization problems on random graphs and random constraint satisfaction problems~\cite{mezard2002analytic,krzakala2007gibbs,decelle2011asymptotic}.  Rigorously proving these physics predictions is a major task in mathematics and theoretical computer science. 

The specific diluted mean-field model we address here  is the anti-ferromagnetic Ising model on sparse random graphs. The Hamiltonian of this model is the number of edges in the random graph cut by the given configuration, and the ground states correspond to the maximum cuts in the graph. Finding the maximum cut in a graph is one of the best-known constraint satisfaction problems, and due to this connection this model has been studied extensively in computer science as well as statistical physics. It is known that as the average degree of the random graph increases, many aspects of this model can be understood via the SK model. Dembo, Montanari, and Sen showed that the typical size of the maximum cut in both Erd\"{o}s-R\'{e}nyi random graphs and random regular graphs can be related to the ground state energy of SK model~\cite{dembo_extremal_2017}. They proved this by interpolating the free energy between the two models (described in more detail below). El Alaoui, Montanari, and Sellke then gave an algorithm for finding a near maximum cut in a random regular graph by adapting Montanari's message-passing algorithm for the SK model~\cite{alaoui_local_2021}. 

It is natural to ask if this connection between the anti-ferromagnetic Ising model and the SK model can be extended further. In particular, does the search vs. sampling phenomenon also arise in the anti-ferromagnetic Ising model on random graphs? In this paper, we give an affirmative answer to this question.

\subsection{Main Results}
\label{secMainResults}

Before stating our main results we  introduce some necessary definitions and notation. We refer to an element in $\{-1, 1\}^n$ as a \emph{configuration}. Given a graph $G$ and its adjacency matrix $A_G$, we define the \emph{Ising Hamiltonian} $H_G(\sigma) = \sum_{1 \leq i \leq j \leq n} (A_G)_{ij}\sigma(i)\sigma(j)$. Note that $\frac{|E(G)| + H_G(\sigma)}{2}$ is exactly the number of edges in the graph whose endpoints are assigned the same spin by $\sigma$. In particular, a ground state (i.e., a configuration that minimizes the Hamiltonian) of $H_G$ corresponds to a maximum cut in the graph. For any \emph{inverse temperature parameter} $\beta \in \mathbb{R}$, $H_G$ induces the \emph{Gibbs distribution}

\begin{equation}
\mu_{\beta, G}(\sigma) = \frac{\exp(-\beta H(\sigma))}{Z_G(\beta)}.
\end{equation}
where
\begin{equation}
 Z_G(\beta)=   \sum_{\sigma'\in \{-1, 1\}^n}\exp(-\beta H(\sigma')) 
\end{equation}
is the \emph{partition function}, the normalizing constant of the Gibbs measure.

In this paper, as is common in prior works, the random graph model we consider will be the Poissonized random multigraph model $\mathbb{G}_{n,d}^\Po$, where we first sample the number of edges from a Poisson distribution with mean $dn/2$ (so that the expected average degree is $d$), and then for each edge sample both of its endpoints independently from $[n] = \{1, 2, \ldots, n\}$. 

Given any two measures $\mu_1, \mu_2$ over $\{-1, 1\}^n$, we define their \emph{normalized Wasserstein distance} to be 
    \begin{equation}
    W_{2,n}(\mu_1, \mu_2) = \inf_{\pi \in \Gamma(\mu_1, \mu_2)} \left(\E_{(\sigma_1, \sigma_2) \sim\pi} \left[\frac{1}{n}\sum_{i = 1}^n(\sigma_1(i) - \sigma_2(i))^2\right]\right)^{1/2},
    \end{equation}
    where $\Gamma(\mu_1, \mu_2)$ is the set of all couplings of $\mu_1$ and $\mu_2$; that is, measures over $\{-1, 1\}^n \times \{-1, 1\}^n$ whose marginal on the first argument is $\mu_1$ and the second $\mu_2$.

Our main result is a hardness of sampling theorem in terms of the $W_{2, n}$ distance for the measure $\mu_{\beta, G}$ against \emph{stable} sampling algorithms where $G \sim \mathbb{G}_{n,d}^\Po$. Informally, a sampling algorithm is stable if its output is insensitive to small perturbations of the input. See Definition~\ref{def:stability_definition} for a formal definition.
\begin{theorem}\label{thm:main}
Let $\{\mathsf{ALG}_n\}_{n \geq 1}$ be a family of randomized sampling algorithms where $\alg_n$ takes as input an $n$-vertex (multi-)graph $G$ and an inverse temperature parameter $\beta$ and produces an output law $\mu^{\alg}_{\beta, G}$ over $\{-1, 1\}^n$. For every $\beta > 1$ and $d$ sufficiently large, if $\{\mathsf{ALG}_n\}_{n \geq 1}$ is \emph{stable} at inverse temperature parameter $\beta / \sqrt{d}$, then
\begin{equation}
\liminf\limits_{n \to \infty}\E_{G \sim \mathbb{G}_{n,d}^\Po}\left[W_{2,n}\left(\mu^{\alg}_{\beta/\sqrt{d}, G},\mu_{\beta/\sqrt{d}, G}\right)\right] > 0.
\end{equation}
\end{theorem}
That is, stable algorithms cannot sample from the model with vanishing error in Wasserstein distance.

It is known that the replica symmetry breaking threshold for $\mathbb{G}^{\Po}_{n,d}$ occurs at the Kesten-Stigum bound $\beta^\dag(d) = \frac{1}{2}\log\frac{\sqrt{d} + 1}{\sqrt{d} - 1}$\footnote{Unless otherwise specified, all logarithms in this paper are natural logarithms.}~\cite{mossel_reconstruction_2015, coja-oghlan_ising_2022}, so the threshold $\beta = 1$ corresponds  exactly to the normalized limit $\lim_{d \to \infty} \sqrt{d} \cdot \beta^\dag(d) = 1$. Note that $\beta = 1$ is also the replica symmetry breaking threshold of the Sherrington-Kirkpatrick model.

As it turns out, the same proof for Theorem~\ref{thm:main} yields sampling hardness for near-maximum and near-minimum bisections as well. Let $A_n = \{\sigma \in \{-1, 1\}^n: |\sum_{i = 1}^n\sigma(i)| \leq 1\}$ be the set of configurations in which the numbers of $+1$s and $-1$s differ by at most one, which we refer to as \emph{bisections}. The Gibbs distribution $\mu_{\beta, G}$ when restricted to $A_n$ gives
\begin{equation}
\mu_{\beta, G}^\bis(\sigma) = \frac{\exp(-\beta H(\sigma))}{Z_G^\bis(\beta)}.
\end{equation}
where 
\begin{equation}
  Z_G^\bis(\beta)  =\sum_{\sigma'\in A_n}\exp(-\beta H(\sigma'))
\end{equation}
is the bisection partition function.
Note that $\mu_{\beta, G}^\bis(\sigma)$ prefers bisections that cut more edges when $\beta > 0$, and bisections with fewer edges when $\beta < 0$.  This model is also known as the zero-magnetization Ising model and has been studied both in the statistical physics literature on disordered systems~\cite{mezard1987mean} and recently in computer science~\cite{carlson2022computational,bauerschmidt2023kawasaki,kuchukova2024fast}.

\begin{theorem}\label{thm:main_bis}
Let $\{\mathsf{ALG}_n\}_{n \geq 1}$ be a family of randomized sampling algorithms where $\alg_n$ takes as input an $n$-vertex (multi-)graph $G$ and an inverse temperature parameter $\beta$ and produces an output law $\mu^{\alg}_{\beta, G}$ over $\{-1, 1\}^n$. For every $\beta \in \mathbb{R}$ such that $|\beta| > 1$ and $d$ sufficiently large, if $\{\mathsf{ALG}_n\}_{n \geq 1}$ is \emph{stable} at inverse temperature parameter $\beta / \sqrt{d}$, then
\begin{equation}
\liminf\limits_{n \to \infty}\E_{G \sim \mathbb{G}_{n,d}^\Po}\left[W_{2,n}\left(\mu^{\alg}_{\beta/\sqrt{d}, G},\mu^\bis_{\beta/\sqrt{d}, G}\right)\right] > 0.
\end{equation}
\end{theorem}

Theorems~\ref{thm:main} and~\ref{thm:main_bis} exhibit a gap between search and sampling for the maximum cut and max/minimum bisection problem: while we have a search algorithm known under the widely believed ``No Overlap Gap'' conjecture~\cite{alaoui_local_2021} \footnote{\cite{alaoui_local_2021} stated the result for random regular graphs, but it can be transferred to other graph models using e.g. the argument in~\cite{chen_local_2023}.}, no stable algorithm can sample from the Gibbs distribution with arbitrary precision in the Wasserstein metric.  Put another way, the algorithm of~\cite{alaoui_local_2021} finds solutions of value $(1-\epsilon)\cdot \text{OPT}$ for any fixed $\epsilon>0$; Theorems~\ref{thm:main} and~\ref{thm:main_bis} rule out stable algorithms that sample solutions of these values approximately uniformly, since a standard reduction of counting to sampling (e.g~\cite{bezakova2008accelerating}) reduces the task of sampling from the Gibbs measure to sampling uniformly from solutions of a  given value.   
\subsection{Techniques}
\label{secTechniques}

To show that sampling from the anti-ferromagnetic Ising  Gibbs distribution on random graphs is hard for stable sampling algorithms, we use the framework that was used by~\cite{alaoui_sampling_2022} to show that sampling from the Sherrington-Kirkpatrick Gibbs distribution is hard. The key quantity in this approach is the \emph{overlap} between two configurations $\sigma_1, \sigma_2$,  defined by  
$$R_{1,2}(\sigma_1, \sigma_2) := \frac{1}{n}\sum_{i = 1}^n \sigma_1(i)\sigma_2(i) \, .$$
The framework consists of the following two steps:
\begin{enumerate}
\item Show that w.h.p. over the disorder, the average of the squared overlap $R_{1,2}(\sigma_1,\sigma_2)^2$ with respect to two independent samples $\sigma_1$ and $\sigma_2$ from the Gibbs distribution is at least some positive constant independent of $n$.  In particular, this is expected to hold when the Gibbs distribution exhibits replica symmetry breaking, which explains the threshold $\beta=1$.
\item Show that the Gibbs distribution exhibits \emph{disorder chaos}. That is, if we perturb the input (the random coupling constants in the case of the SK model or the random graph in the case of the Ising model) slightly (see Section~\ref{secOverview} for a formal definition of this perturbation) and then take a sample $\sigma_1$ from the Gibbs distribution for the original system and a sample $\sigma_2$ from the Gibbs distribution for the perturbed system, with high probability $R_{1,2}(\sigma_1,\sigma_2)$ is very close to zero. 
\end{enumerate}

Combining these two steps and using a connection between $R_{1,2}$ and $W_{2, n}$ distance (see Lemma~\ref{lem:overlap_to_w2}), we have that for arbitrarily small perturbations, the $W_{2,n}$ distance between the original Gibbs distribution and the  perturbed Gibbs distribution is at least some positive constant. However, for stable sampling algorithms, if we make the perturbation sufficiently small then the $W_{2,n}$ distance between the old and new output distributions can also be arbitrarily small. This implies that the output distribution of any stable sampling algorithm will have strictly positive $W_{2,n}$ distance from the target Gibbs distribution. 

In order to carry out this framework for sparse random graph models, we prove new \emph{structural properties} of the  Gibbs measures $\mu_{\beta/\sqrt{d}, G}$ and $\mu^\bis_{\beta/\sqrt{d}, G}$ and connect them to those of the Gibbs measures $\mu_{\beta, \bg}$ and $\mu^\bis_{\beta, \bg}$ of the SK model (see Section~\ref{secOverview} for the formal definition of $\mu_{\beta, \bg}$ and $\mu_{\beta, \bg}^\bis$). This further extends the close connection between the behavior of the SK model and the sparse random graph model. Previously, it was shown by Dembo, Montanari, and Sen~\cite{dembo_extremal_2017} that the \emph{free energy} of the sparse model converges to that of the SK model as the average degree $d$ increases to infinity. We first extend this to the average Hamiltonian of a sample drawn from the Gibbs distribution, showing that this quantity also converges in the same manner, regardless of the bisection restriction. This is done by using convexity of the free energy and a well-known observation that the average Hamiltonian can be obtained by taking derivative of the free energy (see \eqref{eqn:free_energy_to_average_energy} and \eqref{eqn:free_energy_to_average_energy_bis}). The main difficulty here is establishing that the limiting free energy is the same regardless of the bisection restriction, the details of which can be found in Appendix~\ref{section:free_energy_bisection}.

\begin{proposition}\label{prop:average_energy_SK_sparse_sec1} \ 
\begin{enumerate}
    \item[(a)] 
    For every $\beta \geq 0$, we have
    \begin{equation}
    \lim_{d \to \infty}\limsup_{n \to \infty} \frac{1}{n} \left| \E_{\bg}\left[\E_{\sigma \sim \mu_{\beta, \bg}}\left[ H_{\SK}(\sigma) \right]\right] - \frac{1}{\sqrt{d}}\E_{G}\left[\E_{\sigma \sim \mu_{\beta / \sqrt{d}, G}}\left[ H_{G}(\sigma) \right]\right] \right| = 0.
    \end{equation}
    \item[(b)] 
    For every $\beta \in \mathbb{R}$, we have
    \begin{equation}
    \lim_{d \to \infty}\limsup_{n \to \infty} \frac{1}{n} \left| \E_{\bg}\left[\E_{\sigma \sim \mu_{\beta, \bg}^\bis}\left[ H_{\SK}(\sigma) \right]\right] - \frac{1}{\sqrt{d}}\E_{G}\left[\E_{\sigma \sim \mu_{\beta / \sqrt{d}, G}^\bis}\left[ H_{G}(\sigma) \right]\right] \right| = 0.
    \end{equation}
\end{enumerate}
\end{proposition}

 Our main technical contribution is the following proposition which says that as the average degree $d$ goes to infinity, the expected average squared overlap between two independent samples from the Gibbs measure $\mu_{\beta / \sqrt{d}, G}$ converges to those sampled from the Gibbs measure $\mu_{\beta, \bg}$ of the SK model, again regardless of the bisection restriction. This gives the first ingredient in the sampling hardness framework. For the SK Hamiltonian, there is a known connection between the average Hamiltonian and the average squared overlap, which can be established using the Gaussian integration by parts trick (see e.g.~\cite{panchenko_sherrington-kirkpatrick_2013} for more details). In this paper, we show that one can use the Stein-Chen identity for the Poisson distribution (see Lemma~\ref{lemma:stein-chen}) in place of the Gaussian integration by parts trick to establish a similar connection for the sparse model in the large degree limit.

\begin{proposition}\label{prop:overlap_interpolation_sec1} \
    \begin{enumerate}
        \item[(a)] 
        For any $\beta \geq 0$, 
        \begin{equation}
        \lim_{d \to \infty}\limsup_{n \to \infty}\left|\E_\bg\left[\E_{\sigma_1, \sigma_2 \sim \mu_{\beta, \bg}}\left[R_{1,2}(\sigma_1, \sigma_2)^2\right]\right] - \E_G\left[\E_{\sigma_1, \sigma_2 \sim \mu_{\beta / \sqrt{d}, G}}\left[R_{1,2}(\sigma_1, \sigma_2)^2\right]\right]\right| = 0.
        \end{equation}        
        \item[(b)]  
    For any $\beta \in \mathbb{R}$,
        \begin{equation}
        \lim_{d \to \infty}\limsup_{n \to \infty}\left|\E_\bg\left[\E_{\sigma_1, \sigma_2 \sim \mu_{\beta, \bg}^\bis}\left[R_{1,2}(\sigma_1, \sigma_2)^2\right]\right] - \E_G\left[\E_{\sigma_1, \sigma_2 \sim \mu_{\beta / \sqrt{d}, G}^\bis}\left[R_{1,2}(\sigma_1, \sigma_2)^2\right]\right]\right| = 0.
        \end{equation}   
    \end{enumerate}
\end{proposition}

 Finally, to obtain disorder chaos for the sparse model which is the second component in the sampling hardness framework, we use the fact that two configurations sampled from the coupled system for the SK model have nontrivial overlap with only exponentially small probability (see Theorem~\ref{thm:sk_pos_temp_chaos}). This causes a gap in free energy between the coupled system and uncoupled system, which can then be transferred to the sparse model using the interpolation result by~\cite{chen_disorder_2018, chen_suboptimality_2019}. We can then translate this gap back to a disorder chaos statement.

\subsection{Organization}
The paper is organized as follows. In Section~\ref{secOverview} we set up the notation and give a detailed overview of the proof. In Section~\ref{sec:average_energy} and Section~\ref{sec:average_overlap}, we establish the correspondence of average energy and average overlap between the two models. Finally in~Section~\ref{sec:disorder_chaos}, we transfer the disorder chaos property from the SK model to the sparse random graph models.

\section{Overview of the proof}
\label{secOverview}

\paragraph{Gibbs distributions}

For $\sigma \in \{-1, 1\}^n, \bX \in \mathbb{R}^{n \times n}$, we define the Hamiltonian
\[
H(\sigma; \bX) := \sum_{i, j = 1}^n \bX_{ij}\sigma(i)\sigma(j).
\]

For any $\beta \in \mathbb{R}$, it induces the following Gibbs distribution on $\{-1, +1\}^n$:
\[
 \mu_{\beta, \bX}(\sigma) = \frac{ \exp\left(-\beta H(\sigma; \bX)\right)}{Z(\beta, \bX)}, 
\]
where $Z(\beta, \bX) = \sum_{\sigma \in \{-1, 1\}^n} \exp\left(-\beta H(\sigma; \bX)\right)$ is the normalizing factor commonly referred to as the partition function. $\beta$ is commonly referred to as the \emph{inverse temperature} parameter in statistical physics. We also consider the restriction of the above distribution to bisections $\sigma \in A_n$, with
\[
 \mu^\bis_{\beta, \bX}(\sigma) = \frac{ \exp\left(-\beta H(\sigma; \bX)\right)}{Z^\bis(\beta, \bX)}, \quad Z^\bis(\beta, \bX) = \sum_{\sigma \in A_n} \exp\left(-\beta H(\sigma; \bX)\right).
\]
We refer to such restricted versions as bisection models. In statistical physics literature, this is sometimes also called the model with \emph{zero magnetization} and the version without the bisection restriction is called the model with \emph{non-fixed magnetization}. Here, the term magnetization refers to the quantity $m(\sigma) = \sum_{i = 1}^n \sigma(i) / n$. We sometimes also write  $\mu(\sigma; \beta, \bX)$ and $\mu^\bis(\sigma; \beta, \bX)$ instead of $\mu_{\beta, \bX}(\sigma)$ and $\mu^\bis_{\beta, \bX}(\sigma)$ if we wish to stress the dependence on $\beta$ and $\bX$.

$\bX$ is a random matrix drawn from some distribution and in this paper we consider two important cases. In the first case, $\bX = \bg$, where $\bg_{ij} \sim N(0, 1/2n)$ independently. This is known as the Sherrington-Kirkpatrick model. In the second case we have $\bX = \bA$, where each entry $\bA_{ij}$ is an independent $\Po(d/2n)$ random variable for some parameter $d > 0$. This case corresponds to sparse random (multi-)graphs with average degree $d$, where the vertex set is $[n] = \{1, 2, \ldots, n\}$ and the multiplicity of the edge $\{i, j\}$ is $\bA_{ij} + \bA_{ji}$ if $i \neq j$ and $\bA_{ii}$ if $i = j$ (note that here $\bA$ is not the adjacency matrix of the graph and can be asymmetric). Any configuration $\sigma$ can be thought of as the indicator vector for a cut where the vertices assigned 1 by $\sigma$ is on one side and those assigned $-1$ is on the other side, and $H(\sigma; \bA)$ is equal to the difference between the number of edges crossing the cut and the number of edges not crossing the cut. When $\beta > 0$, the Gibbs measure prefers those configurations that cut more edges, and this corresponds to the \emph{Maximum Cut} problem for the non-fixed magnetization model and the \emph{Maximum Bisection} problem for the zero-magnetization model. This case is known as the anti-ferromagnetic Ising model in the statistical physic literature. When $\beta < 0$, the non-fixed magnetization model corresponds to the \emph{Minimum Cut} problem while the zero-magnetization model corresponds to the \emph{Minimum Bisection} problem. This case is known as the ferromagnetic Ising model in the statistical physic literature.

For notational convenience, we sometimes drop the random matrix and write $H_{\SK}(\sigma) :=  H(\sigma; \bg)$ and $H_d(\sigma) := H(\sigma; \bA)$. We sometimes refer to these two models as \emph{dense} or \emph{sparse} models respectively.

\paragraph{Gibbs average} For any function $f: \{-1, 1\}^n \to \mathbb{R}$, we can define its average with respect to $\mu$ or $\mu^\bis$: 
\[
\langle f(\sigma) \rangle_{\beta, \bX} = \sum_{\sigma \in \{-1, 1\}^n}f(\sigma)\mu_{\beta, \bX}(\sigma), \quad \langle f(\sigma) \rangle_{\beta, \bX}^\bis = \sum_{\sigma \in A_n}f(\sigma)\mu^\bis_{\beta, \bX}(\sigma).
\]
 We can extend this definition to functions that take multiple configurations as inputs, in which case we assume that the configurations are sampled independently from the Gibbs distribution:
\[
\langle f(\sigma_1, \ldots, \sigma_k) \rangle_{\beta, \bX} = \sum_{\sigma_1, \ldots, \sigma_k \in \{-1, 1\}^n}\mu_{\beta, \bX}(\sigma_1)\cdots\mu_{\beta, \bX}(\sigma_k)f(\sigma_1, \ldots, \sigma_k), \quad f: \{-1, 1\}^{n\times k} \to \mathbb{R},
\]
and the average over $\mu^\bis$ is defined similarly. We sometimes drop the subscripts if they are clear from context. One function of particular interest to us is the \emph{overlap} between two configurations, defined as
\[
R_{1,2}(\sigma_1, \sigma_2)= \frac{1}{n}\sum_{i = 1}^n \sigma_1(i)\sigma_2(i).
\]
This function gives the normalized inner product between two configurations. If $\langle |R_{1,2}(\sigma_1, \sigma_2)| \rangle_{\beta,\bX}$ is close to zero, then two configurations independently sampled from the Gibbs distribution $\mu_{\beta, \bX}(\sigma)$ are nearly orthogonal.

\paragraph{Free energy}
The free energy for these models is defined as follows.
\[
\Phi_{n,\SK}(\beta) = \frac{1}{n} \E_{\bg} \left[\log Z(\beta, \bg)\right], \quad \Phi_{n,d}(\beta) = \frac{1}{n} \E_{\bA} \left[\log Z(\beta, \bA)\right].
\]
We can also define free energy in a similar way when restricted to bisections.
\[
\Phi_{n,\SK}^{\bis}(\beta) = \frac{1}{n} \E_{\bg} \left[\log Z^{\bis}(\beta, \bg)\right], \quad \Phi_{n,d}^{\bis}(\beta) = \frac{1}{n} \E_{\bA} \left[\log Z^{\bis}(\beta, \bA)\right].
\]

The following proposition is a simple consequence of Cauchy-Schwarz.
\begin{proposition}
    Both $\log Z(\beta, \bX)$ and $\log Z^\bis(\beta, \bX)$ are convex in $\beta$.
\end{proposition}
\begin{proof}
    For any $\beta_1, \beta_2$ we have
    \begin{align*}
    Z\left(\frac{\beta_1 + \beta_2}{2}, \bX\right) & = \sum_{\sigma \in \{-1, 1\}^n} \exp\left(-\frac{\beta_1 + \beta_2}{2} H(\sigma; \bX)\right) \\
    & \leq \sqrt{\left( \sum_{\sigma \in \{-1, 1\}^n} \exp\left(-\beta_1 H(\sigma; \bX)\right)\right)\left( \sum_{\sigma \in \{-1, 1\}^n} \exp\left(-\beta_2 H(\sigma; \bX)\right)\right)} \\
    & = \sqrt{Z(\beta_1, \bX)Z(\beta_2, \bX)}. \\
    \end{align*}
    Taking the logarithm of both sides, we get that $\log Z(\beta, \bX)$ is convex in $\beta$. The convexity of $\log Z^\bis(\beta, \bX)$ can be shown in the same way.
\end{proof}

By taking expectation over the randomness of $\bX$ we immediately obtain the following corollary.
\begin{corollary}
    $\Phi_{n,\SK}(\beta), \Phi_{n,\SK}^{\bis}(\beta), \Phi_{n,d}(\beta), \Phi_{n,d}^{\bis}(\beta)$ are all convex in $\beta$.
\end{corollary}

\paragraph{Correspondence between dense and sparse models}

It is known that as the average degree $d$ gets larger, the sparse model will ``converge'' to the dense model in some sense. Dembo, Montanari, and Sen proved this for the free energy of these two models using a clever interpolation argument:
\begin{lemma}[\cite{dembo_extremal_2017}]\label{lem:free_energy_interpolation}
    There exist constants $c_1, c_2, c_3 > 0$ independent of $n, \beta, d$ such that for any $\beta \in \mathbb{R}$, 
    \begin{equation}
    \left|\Phi_{n,d}^\bis \left(\frac{\beta}{\sqrt{d}}\right) - \Phi_{n,\SK}^\bis(\beta)\right| \leq c_1  \frac{|\beta|^3}{\sqrt{d}} + c_2 \frac{\beta^4}{d} + c_3 |\beta|\sqrt{d} \cdot \frac{1}{n^2}.\footnote{The $c_3 |\beta|\sqrt{d} \cdot \frac{1}{n^2}$ term is actually given by $c_3 \sqrt{d} \cdot \langle m(\sigma )^2\rangle_t$ where $m(\sigma ) = \frac{1}{n}\sum_{i}\sigma_i$ is the \emph{average magnetization} and $\langle \cdot \rangle_t$ denotes the expectation with respect to some Gibbs measure interpolating between $\mu^\bis(\sigma; \beta, \bg)$ and $\mu^\bis(\sigma; \beta/\sqrt{d}, \bA)$. Since $\sigma$ is a bisection, this term is 0 when $n$ is even and $1/n^2$ when $n$ is odd. In~\cite{dembo_extremal_2017} the authors assumed that $n$ is even for simplicity. }
    \end{equation}
    
\end{lemma}

In this paper, we further extend this correspondence to the average energy $\langle H(\sigma)\rangle$ as well as the average overlap $\langle R_{1,2}(\sigma_1, \sigma_2)^2 \rangle$.
\begin{proposition}[Restatement of Proposition~\ref{prop:average_energy_SK_sparse_sec1}] \label{lem:average_energy_SK_sparse}\ 
\begin{enumerate}
    \item[(a)] 
    For every $\beta \geq 0$, we have
    \begin{equation}
    \lim_{d \to \infty}\limsup_{n \to \infty} \frac{1}{n} \left| \E_{\bg}[\langle H_{\SK}(\sigma) \rangle_{\beta, \bg}] - \frac{1}{\sqrt{d}}\E_{\bA}[\langle H_{d}(\sigma) \rangle_{\beta/\sqrt{d}, \bA}] \right| = 0.
    \end{equation}
    \item[(b)] 
    For every $\beta \in \mathbb{R}$, we have
    \begin{equation}
    \lim_{d \to \infty}\limsup_{n \to \infty} \frac{1}{n} \left| \E_{\bg}[\langle H_{\SK}(\sigma) \rangle_{\beta, \bg}^\bis] - \frac{1}{\sqrt{d}}\E_{\bA}[\langle H_{d}(\sigma) \rangle_{\beta/\sqrt{d}, \bA}^\bis] \right| = 0.
    \end{equation}
\end{enumerate}
\end{proposition}

\begin{proposition}[Restatement of Proposition~\ref{prop:overlap_interpolation_sec1}] \label{lem:overlap_interpolation} \
    \begin{enumerate}
        \item[(a)] 
        For any $\beta \geq 0$, 
        \begin{equation}
        \lim_{d \to \infty}\limsup_{n \to \infty}\left|\E_\bg[\langle R_{1,2}^2\rangle_{\beta, \bg}] - \E_\bA[\langle R_{1,2}^2\rangle_{\beta / \sqrt{d}, \bA}]\right| = 0.
        \end{equation}        
        \item[(b)]  
    For any $\beta \in \mathbb{R}$,
    \begin{equation}
        \lim_{d \to \infty}\limsup_{n \to \infty}\left|\E_\bg[\langle R_{1,2}^2\rangle_{\beta, \bg}^\bis] - \E_\bA[\langle R_{1,2}^2\rangle_{\beta / \sqrt{d}, \bA}^\bis]\right| = 0.
    \end{equation}
    \end{enumerate}
\end{proposition}

We prove Proposition~\ref{lem:average_energy_SK_sparse} in Section~\ref{sec:average_energy} and Proposition~\ref{lem:overlap_interpolation} in Section~\ref{sec:average_overlap}.

\paragraph{Disorder chaos and stable algorithms}

Disorder chaos is a well-studied phenomenon in spin glass theory. The term ``disorder'' refers to the random matrix $\bX$, and ``chaos'' describes what happens to the Gibbs distribution if we slightly perturb $\bX$. For the SK model,  consider the following notion of perturbation. Let $\bg, \bg'$ be two independent copies of Gaussian matrices where $\bg_{ij}, \bg_{ij}' \sim N(0, 1/2n)$ independently for each $i, j$. Given any perturbation parameter $t \geq 0$, we can consider the two measures $\mu_{\beta, \bg}(\sigma)$ and  $\mu_{\beta, \bg_t}(\sigma)$ where $\bg_t = \sqrt{1 - t}\bg + \sqrt{t}\bg'$. Define $\langle f(\sigma_1, \sigma_2) \rangle_{\beta, \bg, \bg_t}$ to be the average of $f(\sigma_1, \sigma_2)$ where $\sigma_1 \sim \mu_{\beta, \bg}$ and $\sigma_2 \sim \mu_{\beta, \bg_t}$ sampled independently from these two distributions respectively, i.e.,
\[
\langle f(\sigma_1, \sigma_2) \rangle_{\beta, \bg, \bg_t} = \sum_{\sigma_1, \sigma_2 \in \{-1, 1\}^n}f(\sigma_1, \sigma_2)\mu_{\beta, \bg}(\sigma_1)\mu_{\beta, \bg_t}(\sigma_2).
\]

It is known that two such samples $\sigma_1, \sigma_2$ will likely have some nontrivial overlap when $t = 0$, in which case there is no perturbation and $\langle f(\sigma_1, \sigma_2) \rangle_{\beta, \bg, \bg_t}$ is simply $\langle f(\sigma_1, \sigma_2) \rangle_{\beta, \bg}$.
\begin{theorem}[See e.g.~\cite{alaoui_sampling_2022}]\label{thm:sk_no_perturbation}
    If $|\beta| > 1$, then there exists $\epsilon = \epsilon(\beta) > 0$ such that
    \begin{equation}\label{eqn:overlap_sk_nonzero}
    \lim_{n \to \infty}\E_{\bg}\left[\langle R_{1,2}(\sigma_1, \sigma_2)^2 \rangle_{\beta, \bg}\right] \geq \epsilon(\beta).
    \end{equation}
\end{theorem}
However, when $t > 0$, then the overlap will be zero in the $n \to \infty$ limit.

\begin{theorem}[\cite{chatterjee_disorder_2009}]\label{thm:sk_chaos_chatterjee}
    For all $\beta \in \mathbb{R}, t \in (0, 1]$, we have   \begin{equation}\label{eqn:sk_disorder_chaos}
        \lim_{n \to \infty} \E_{\bg,\bg'}\left[\langle R_{1,2}(\sigma_1, \sigma_2)^2 \rangle_{\beta, \bg,\bg_t}\right] = 0.
    \end{equation}
\end{theorem}

If we think of the expression on the left hand side of \eqref{eqn:overlap_sk_nonzero} and \eqref{eqn:sk_disorder_chaos} as a function of $t$, then \eqref{eqn:overlap_sk_nonzero} and \eqref{eqn:sk_disorder_chaos} together suggest that this function is not right-continuous at $t = 0$ if $|\beta| > 1$. It was shown in~\cite{alaoui_sampling_2022} that this property very naturally obstructs a class of sampling algorithms for $\mu_{\beta, \bg}$.

\begin{definition}[Definition 2.2 in~\cite{alaoui_sampling_2022}]\label{def:stability_definition}
    Let $\{\alg_n\}_{n \geq 1}$ be a family of sampling algorithms for the Gibbs measure $\mu_{\beta, \bg}$, where $\alg_n$ takes an $n\times n$ matrix $\bg$ and $\beta \in \mathbb{R}$ as input and outputs an assignment in $\{-1, 1\}^n$. Let $\mu_{\beta, \bg}^{\alg}$ be the output law of $\alg_n$. We say that $\{\alg_n\}_{n \geq 1}$ is stable (with respect to disorder at $\beta$), if 
    \begin{equation}\label{eqn:stability_definition}
    \lim_{t \to 0} \limsup_{n \to \infty} \E\left[W_{2, n}(\mu_{\beta, \bg}^\alg, \mu_{\beta, \bg_t}^{\alg})\right] = 0.
    \end{equation}
\end{definition}

Here $\Gamma(\mu_1, \mu_2)$ is the set of distributions over pairs of configurations (i.e. $\{-1, 1\}^n \times \{-1, 1\}^n$) such that its marginal on the first configuration is $\mu_1$ and its marginal on the second configuration is $\mu_2$. Intuitively, \eqref{eqn:stability_definition} means that stable algorithms are not able to make the leap at the discontinuity $t = 0$ implied by \eqref{eqn:overlap_sk_nonzero} and \eqref{eqn:sk_disorder_chaos}, and therefore must be producing a distribution that is bounded away from the Gibbs distribution $\mu_{\beta, \bg}$ in terms of the $W_{2,n}$ distance. This intuition is formalized in~\cite{alaoui_sampling_2022} as the following theorem.
\begin{theorem}[Theorem 2.6 in~\cite{alaoui_sampling_2022}]
    Fix $\beta$ such that $|\beta| > 1$. Let $\{\alg_n\}_{n \geq 1}$ be a family of sampling algorithms for the Gibbs measure $\mu_{\beta, \bg}$ that is stable with respect to disorder at $\beta$, then
    \begin{equation}
        \liminf_{n \to \infty} \E_{\bg} \left[W_{2,n}(\mu_{\beta, \bg}, \mu_{\beta, \bg}^\alg)\right] > 0.
    \end{equation}
\end{theorem}

In this paper, we obstruct stable sampling algorithms for sparse models as well. For the sparse models, we consider a slightly different notion of perturbation. Fix the average degree $d > 0$ and the perturbation parameter $t \in [0, 1]$, we will take three independent random matrices $\bA^{(1-t)}, \bA^{(t, 1)}, \bA^{(t, 2)}$ such that $\bA^{(1-t)}_{ij} \sim \Po((1-t)d/(2n))$, $\bA^{(t, 1)}_{ij}, \bA^{(t, 2)}_{ij} \sim \Po(td/(2n))$ independently for all $i,j \in [n]$. We then define $\bA = \bA^{(1-t)} + \bA^{(t, 1)}$ and $\bA_t = \bA^{(1-t)} + \bA^{(t, 2)}$. Similarly to Definition~\ref{def:stability_definition}, we say that a family of sampling algorithms $\{\alg_n\}_{n \geq 1}$ with inputs $\beta$ and $\bA$ is stable at inverse temperature $\beta$ if

    \begin{equation}\label{eqn:stability_definition_sparse}
    \lim_{t \to 0} \limsup_{n \to \infty} \E\left[W_{2, n}(\mu_{\beta, \bA}^\alg, \mu_{\beta, \bA_t}^\alg)\right] = 0.
    \end{equation}

When $t=0$, we have by Proposition~\ref{lem:overlap_interpolation} and Theorem~\ref{thm:sk_no_perturbation} that for $\beta > 1$,
\begin{equation}\label{eq:R12_A}
    \lim_{d \to \infty}\liminf_{n \to \infty}\E_\bg[\langle R_{1,2}^2\rangle_{\beta/\sqrt{d}, \bA}] > 0,
\end{equation}
and for $\beta$ with $|\beta| > 1$
\begin{equation}   
    \lim_{d \to \infty}\liminf_{n \to \infty}\E_\bg[\langle R_{1,2}^2\rangle_{\beta/\sqrt{d}, \bA}^\bis] > 0.
\end{equation}
On the other hand, if $t > 0$, we show in the following lemma that the overlap tends to zero as $n$ goes to infinity.

\begin{lemma}\label{lem:disorder_chaos_sparse}
    Fix $t > 0$.
    \begin{enumerate}
        \item[(a)] For any $\beta \geq 0$,
    \begin{equation}\label{eq:R12_A_At}
    \lim_{d \to \infty}\limsup_{n \to \infty}\E[\langle R_{1,2}^2\rangle_{\beta/\sqrt{d},\bA,\bA_t}] = 0.
    \end{equation}
        \item[(b)] For any $\beta \in \mathbb{R}$,
    \begin{equation}
    \lim_{d \to \infty}\limsup_{n \to \infty}\E[\langle R_{1,2}^2\rangle_{\beta/\sqrt{d},\bA,\bA_t}^\bis] = 0.
    \end{equation}
    \end{enumerate}
\end{lemma}
The following lemma establishes the connection between the overlap and the $W_{2,n}$ distance.
\begin{lemma}[Lemma 5.3 in~\cite{alaoui_sampling_2022}]\label{lem:overlap_to_w2}
    Fix $n \in \mathbb{N}$. Let $\mu_1, \mu_2, \nu_1, \nu_2$ be distributions over $\{-1, +1\}^n$. We have
    \begin{equation}
    \left|\E_{(\sigma, \sigma')\sim \mu_1 \otimes \nu_1}\left[|R_{1,2}(\sigma, \sigma')|\right] - \E_{(\sigma, \sigma')\sim \mu_2 \otimes \nu_2}\left[|R_{1,2}(\sigma, \sigma')|\right]\right| \leq W_{2,n}(\mu_1, \mu_2) + W_{2, n}(\nu_1, \nu_2).
    \end{equation}
\end{lemma}

We now have the ingredients needed for proving Theorem~\ref{thm:main}. Theorem~\ref{thm:main_bis} can be proved in the same way and we omit the details here.

\begin{proof}[Proof of Theorem~\ref{thm:main}]
We drop the inverse temperature parameter in the $\langle \cdot \rangle$ notation for brevity. For any arbitrary $t > 0$, by choosing $\mu_1 = \mu_2 = \nu_1 = \mu_{\beta/\sqrt{d},\bA}$ and $\nu_2 = \mu_{\beta/\sqrt{d},\bA_t}$ in Lemma~\ref{lem:overlap_to_w2}, we get that
\begin{align}
    W_{2,n}(\mu_{\beta/\sqrt{d},\bA}, \mu_{\beta/\sqrt{d},\bA_t}) \geq \left|
    \left\langle|R_{1,2}(\sigma, \sigma')|\right\rangle_{\bA} - \left\langle|R_{1,2}(\sigma, \sigma')|\right\rangle_{\bA, \bA_t}\right|.
\end{align}

Taking expectation on both sides, we obtain
\begin{align*}
    \E_{\bA,\bA_t}\left[W_{2,n}(\mu_{\beta/\sqrt{d},\bA}, \mu_{\beta/\sqrt{d},\bA_t})\right] & \geq \E_{\bA,\bA_t}\left[\left|
    \left\langle|R_{1,2}(\sigma, \sigma')|\right\rangle_{\bA} - \left\langle|R_{1,2}(\sigma, \sigma')|\right\rangle_{\bA, \bA_t}\right|\right] \\ 
    & \geq \E_{\bA,\bA_t}\left[
    \left\langle|R_{1,2}(\sigma, \sigma')|\right\rangle_{\bA}\right] - \E_{\bA,\bA_t}\left[\left\langle|R_{1,2}(\sigma, \sigma')|\right\rangle_{\bA, \bA_t}\right].
\end{align*}
In the second inequality we used the fact that $|a - b| \geq |a| - |b|$ for any $a, b \in \mathbb{R}$. Taking $\liminf$ on both sides, we obtain
\begin{equation}
    \liminf_{n \to \infty}\E_{\bA,\bA_t}\left[W_{2,n}(\mu_{\beta/\sqrt{d},\bA}, \mu_{\beta/\sqrt{d},\bA_t})\right] \geq \liminf_{n \to \infty}\E_{\bA,\bA_t}\left[
    \left\langle|R_{1,2}(\sigma, \sigma')|\right\rangle_{\bA}\right] - \limsup_{n \to \infty}\E_{\bA,\bA_t}\left[\left\langle|R_{1,2}(\sigma, \sigma')|\right\rangle_{\bA, \bA_t}\right].
\end{equation}
By \eqref{eq:R12_A} and \eqref{eq:R12_A_At}, if $d$ is sufficiently large, for some $\epsilon = \epsilon(\beta) > 0$ independent of $t$,

\begin{equation}
    \liminf_{n \to \infty}\E_{\bA,\bA_t}\left[W_{2,n}(\mu_{\beta/\sqrt{d},\bA}, \mu_{\beta/\sqrt{d},\bA_t})\right] \geq \epsilon.
\end{equation}

By the triangle inequality, we have
\begin{equation}
W_{2,n}(\mu_{\beta/\sqrt{d},\bA}, \mu_{\beta/\sqrt{d},\bA_t}) \leq W_{2,n}(\mu_{\beta/\sqrt{d},\bA}, \mu_{\beta/\sqrt{d},\bA}^\alg) + W_{2,n}(\mu_{\beta/\sqrt{d},\bA}^\alg, \mu_{\beta/\sqrt{d},\bA_t}^\alg) + W_{2,n}(\mu_{\beta/\sqrt{d},\bA_t}, \mu_{\beta/\sqrt{d},\bA_t}^\alg).
\end{equation}

Since $\bA$ and $\bA_t$ have the same distribution, we have
\begin{equation}
    \E\left[W_{2,n}(\mu_{\beta/\sqrt{d},\bA}, \mu_{\beta/\sqrt{d},\bA}^\alg)\right] = \E\left[W_{2,n}(\mu_{\beta/\sqrt{d},\bA_t}, \mu_{\beta/\sqrt{d},\bA_t}^\alg)\right].
\end{equation}
Since $\{\alg_n\}_n$ is stable at any inverse temperature, by \eqref{eqn:stability_definition_sparse} and taking $t$ sufficiently small we have
\begin{equation}
    \limsup_{n\to\infty}\E\left[W_{2,n}(\mu_{\beta/\sqrt{d},\bA}^{\alg},\mu_{\beta/\sqrt{d},\bA_t}^{\alg})\right] \leq \frac{\epsilon}{2}.
\end{equation}
It follows that
\begin{equation}
\liminf\limits_{n \to \infty}\E\left[W_{2,n}(\mu_{\beta/\sqrt{d},\bA}^{\alg}, \mu_{\beta/\sqrt{d},\bA})\right] \geq \frac{\epsilon}{4}. 
\end{equation}
\qedhere
\end{proof}

\section{Interpolating the average energy between SK and the diluted model}\label{sec:average_energy}

In this section we prove Proposition~\ref{lem:average_energy_SK_sparse} and establish the interpolation of average energy between sparse and dense models. By Lemma~\ref{lem:free_energy_interpolation}, we already know that the free energy of the sparse model converges to that of the dense model as the average degree $d$ goes to infinity. The following connection between free energy and average energy is well-known:
    \begin{equation}
        \frac{\partial}{\partial \beta} \frac{1}{n} \E\left[ \log Z(\beta, \bX) \right]=\frac{1}{n} \E\left[ \frac{\sum_{\sigma \in \{-1, 1\}^n}-H(\sigma;\bX)\exp(-\beta H(\sigma;\bX))}{Z(\beta, \bX)}\right] =  -\frac{1}{n} \E\left[ \langle H(\sigma;\bX)\rangle_{\beta, \bX}\right].\label{eqn:free_energy_to_average_energy}
    \end{equation}
Similarly,
    \begin{equation}
        \frac{\partial}{\partial \beta} \frac{1}{n} \E\left[ \log Z^\bis(\beta, \bX) \right]=\frac{1}{n} \E\left[ \frac{\sum_{\sigma \in A_n}-H(\sigma;\bX)\exp(-\beta H(\sigma;\bX))}{Z^\bis(\beta, \bX)}\right] =  -\frac{1}{n} \E\left[ \langle H(\sigma;\bX)\rangle_{\beta, \bX}^\bis\right].\label{eqn:free_energy_to_average_energy_bis}
    \end{equation}
To obtain convergence of the partial derivatives, we use the following elementary fact sometimes known as Griffith's lemma in statistical physics.
\begin{proposition}[See e.g.~\cite{talagrand_mean_2011}]\label{prop:griffiths_lemma}
    Let $(f_\alpha)_{\alpha \geq 0}$ be a family of convex and differentiable functions that converges pointwise in an open interval $I$ to $f$, then $\lim_{\alpha \to \infty} f'_\alpha(x) = f'(x)$ at every $x \in I$ where $f'(x)$ exists. 
\end{proposition}

\begin{proof}[Proof of Proposition~\ref{lem:average_energy_SK_sparse}]
    Proposition~\ref{lem:average_energy_SK_sparse} says that 
    \item[(a)] 
    For every $\beta \geq 0$, $\lim_{d \to \infty}\limsup_{n \to \infty} \frac{1}{n} \left| \E_{\bg}[\langle H_{\SK}(\sigma) \rangle_{\beta, \bg}] - \frac{1}{\sqrt{d}}\E_{\bA}[\langle H_{d}(\sigma) \rangle_{\beta/\sqrt{d}, \bA}] \right| = 0.$
    \item[(b)] 
    For every $\beta \in \mathbb{R}$, 
    $\lim_{d \to \infty}\limsup_{n \to \infty} \frac{1}{n} \left| \E_{\bg}[\langle H_{\SK}(\sigma) \rangle_{\beta, \bg}^\bis] - \frac{1}{\sqrt{d}}\E_{\bA}[\langle H_{d}(\sigma) \rangle_{\beta/\sqrt{d}, \bA}^\bis] \right| = 0.$
    
    We first prove part (b).
    Assume for the sake of contradiction that for some $\beta_0 \in \mathbb{R}$
    \begin{equation}
    \limsup_{d \to \infty}\limsup_{n \to \infty} \frac{1}{n} \left| \E_{\bg}[\langle H_{\SK}(\sigma) \rangle_{\beta_0, \bg}^\bis] - \frac{1}{\sqrt{d}}\E_{\bA}[\langle H_{d}(\sigma) \rangle_{\beta_0/\sqrt{d}, \bA}^\bis] \right| = \epsilon > 0.
    \end{equation}
    It follows that we can choose a sequence of pairs $(d_i, n_i)_{i \in \mathbb{N}}$ such that
    \begin{equation}
    \lim_{i \to \infty} \frac{1}{n_i} \left| \E_{\bg}[\langle H_{\SK}(\sigma) \rangle_{\beta_0, \bg}^\bis] - \frac{1}{\sqrt{d_i}}\E_{\bA}[\langle H_{d_i}(\sigma) \rangle_{\beta_0/\sqrt{d_i}, \bA}^\bis] \right| = \epsilon. \label{eqn:lim_H_epsilon}
    \end{equation}
    Note that for every $i$, $n_i$ can be chosen to be sufficiently large so that the $o_n(1)$ term in Lemma~\ref{lem:free_energy_interpolation} is $\leq 1/d_i$, and under this assumption we have $\left(\Phi_{n_i,d_i}^\bis \left(\frac{\beta}{\sqrt{d_i}}\right)\right)_{i \in \mathbb{N}}$ as a family of functions of $\beta$ converges pointwise\footnote{The convergence is in fact uniform in any bounded interval.} to $\lim_{i \to \infty}\Phi_{n_i,\SK}^\bis \left(\beta\right)$. 
    Here we remark that it is known that the limit $\lim_{n \to \infty} \Phi_{n,\SK}(\beta)$ exists and is differentiable in $\beta$~\cite{talagrand_parisi_2006,talagrand_parisi_2006-1}, and by Corollary~\ref{cor:SK_bis_free_energy} we have 
    \begin{equation}
    \lim_{n \to \infty}\Phi_{n,\SK} \left(\beta\right) = \lim_{n \to \infty}\Phi_{n,\SK}^\bis \left(\beta\right) = \lim_{i \to \infty}\Phi_{n_i,\SK}^\bis \left(\beta\right).
    \end{equation} 
    By Proposition~\ref{prop:griffiths_lemma}, we have
    \begin{equation}
    \lim_{i \to \infty} \frac{\partial}{\partial \beta} \Phi_{n_i,d_i}^\bis \left(\frac{\beta}{\sqrt{d_i}}\right) = \frac{\partial}{\partial \beta} \lim_{n \to \infty} \Phi_{n,\SK}^\bis(\beta) = \lim_{n \to \infty}\frac{\partial}{\partial \beta}  \Phi_{n,\SK}^\bis (\beta)= \lim_{i \to \infty}\frac{\partial}{\partial \beta}  \Phi_{n_i,\SK}^\bis(\beta) \label{eqn:partial_limit}
    \end{equation}
    for every $\beta \in \mathbb{R}$. By \eqref{eqn:free_energy_to_average_energy_bis}, we have
\begin{equation}
    \frac{\partial }{\partial \beta} \Phi_{n,d}^\bis \left(\frac{\beta}{\sqrt{d}}\right) = - \frac{1}{\sqrt{d}\cdot n}\E_{\bA}\left[\langle H_{d}(\sigma) \rangle_{\beta/\sqrt{d}, \bA}^\bis\right],\label{eqn:partial_limit_sparse}
\end{equation}
and 
\begin{equation}
\frac{\partial}{\partial \beta}\Phi_{n,\SK}^\bis(\beta) = -\frac{1}{n} \E_{\bg}[\langle H_{\SK}(\sigma) \rangle_{\beta, \bg}^\bis]. \label{eqn:partial_limit_SK}
\end{equation}
\eqref{eqn:partial_limit}, \eqref{eqn:partial_limit_sparse}, and \eqref{eqn:partial_limit_SK} imply that
\begin{equation}
    \lim_{i \to \infty} \frac{1}{n_i} \left| \E_{\bg}[\langle H_{\SK}(\sigma) \rangle_{\beta_0, \bg}^\bis] - \frac{1}{\sqrt{d_i}}\E_{\bA}[\langle H_{d_i}(\sigma) \rangle_{\beta_0/\sqrt{d_i}, \bA}^\bis] \right| = 0.
\end{equation}
However, this contradicts  \eqref{eqn:lim_H_epsilon}, so we obtain part (b).

For part (a), note that by Theorem~\ref{theorem:phi_sparse_limit_bis}, for any bounded open interval $I \subset [0, +\infty)$,
we can choose the sequence $(d_i, n_i)_{i \in \mathbb{N}}$ to satisfy the additional requirement that 
\begin{equation}
    \lim_{i \to \infty}\Phi_{n_i, d_i}^\bis(\beta/ \sqrt{d_i}) = \lim_{i \to \infty}\Phi_{n_i, d_i}(\beta/ \sqrt{d_i})
\end{equation}

for all $\beta \in I$. In particular, we can choose $I$ to include any desired $\beta_0 > 0$ (when $\beta = 0$ the Gibbs measure is uniform and part (a) trivially holds). We can then invoke Proposition~\ref{prop:griffiths_lemma} again and proceed with the same proof.
\end{proof}

\section{Interpolating the average overlap between SK and the diluted model}\label{sec:average_overlap}

In this section we prove Proposition~\ref{lem:overlap_interpolation}. Since the proofs are mostly the same, we will only prove item (a) in Proposition~\ref{lem:overlap_interpolation}. We first state the following well-known fact which relates the average energy to the average overlap of the Gibbs distribution in the SK model.

\begin{lemma}[See e.g. \cite{panchenko_sherrington-kirkpatrick_2013}]\label{lemma:SK_overlap_energy}
For every $\beta \in \mathbb{R}$, we have 
\begin{equation}
\frac{1}{n}\E_\bg\left[\langle H_{\SK}(\sigma) \rangle \right] = \frac{\beta}{2}\E_\bg[1 - \langle R_{1,2}^2\rangle_{\beta, \bg}].
\end{equation}
\end{lemma}

The proof of Lemma~\ref{lemma:SK_overlap_energy} uses the following proposition.

\begin{proposition}\label{prop:overlap_partial_mu}
    Let $\bX \in \mathbb{R}^{n \times n}, \beta \in \mathbb{R}$. We have 
    \begin{equation}
    \sum_{\sigma \in \{-1, 1\}^n}\sum_{i,j = 1}^n\sigma(i)\sigma(j)\frac{\partial }{\partial \bX_{ij}}\mu(\sigma; \beta, \bX) = -\beta n^2 \left(1 - \langle R_{1,2}^2\rangle_{\beta, \bX}\right)
    \end{equation}
\end{proposition}
\begin{proof}
For every $\sigma \in \{-1, 1\}^n$ and $i, j \in [n]$, we have
\begin{align}
& \frac{\partial }{\partial \bX_{ij}}\mu(\sigma; \beta, \bX) \nonumber \\
=\, & \frac{\partial }{\partial \bX_{ij}}\left(\frac{ \exp\left(-\beta H(\sigma; \bX)\right)}{Z(\beta, \bX)}\right)\nonumber \\
=\, & (-\beta) \left( \frac{\sigma(i)\sigma(j)\exp\left(-\beta H(\sigma; \bX)\right)}{Z(\beta, \bX)} - \frac{\sum_{\sigma' \in \{-1, 1\}^n}\sigma'(i)\sigma'(j)\exp\left(-\beta H(\sigma; \bX)\right)\exp\left(-\beta H(\sigma'; \bX)\right)}{Z(\beta, \bX)^2} \right) \nonumber \\
=\, & (-\beta) \left(\sigma(i)\sigma(j)\mu(\sigma; \beta, \bX) - \sum_{\sigma' \in \{-1, 1\}^n}\sigma'(i)\sigma'(j)\cdot\mu(\sigma; \beta, \bX)\mu(\sigma'; \beta, \bX)\right) . \label{eqn:mu_partial}
\end{align}
Summing over all $\sigma, i, j$, we obtain
\begin{align*}
    &\sum_{\sigma \in \{-1, 1\}^n}\sum_{i,j = 1}^n\sigma(i)\sigma(j)\frac{\partial }{\partial \bX_{ij}}\mu(\sigma; \beta, \bX) \\
    =\, & \sum_{\sigma \in \{-1, 1\}^n}\sum_{i,j = 1}^n\sigma(i)\sigma(j)\cdot(-\beta) \left(\sigma(i)\sigma(j)\mu(\sigma; \beta, \bX) - \sum_{\sigma' \in \{-1, 1\}^n}\sigma'(i)\sigma'(j)\cdot\mu(\sigma; \beta, \bX)\mu(\sigma'; \beta, \bX)\right)\\
    =\, & -\beta n^2 \left(1 - \langle R_{1,2}^2\rangle_{\beta, \bX}\right). \qedhere
\end{align*}
\end{proof}

Lemma~\ref{lemma:SK_overlap_energy} can be proved using Proposition~\ref{prop:overlap_partial_mu} and the Gaussian integration by parts formula. We prove a statement similar to Lemma~\ref{lemma:SK_overlap_energy}, but for the sparse model. We will use the following lemma, known as the Stein-Chen identity for the Poisson distribution, in place of Gaussian integration by parts. 

\begin{lemma}[\cite{chen_poisson_1975}]\label{lemma:stein-chen}
    Let $X \sim \Po(\lambda)$. For any bounded function $f$, we have
    \begin{equation}
    \E[X\cdot f(X)] = \lambda \E[f(X + 1)].
    \end{equation}
\end{lemma}

\begin{lemma}\label{lemma:taylor_series_for_mu}
    Let $\bJ_{ij}$ be the matrix with 1 in the $(i, j)$-th entry and 0 everywhere else. Then there exists $\xi \in [0, 1]$ such that
    \begin{equation}
    \left|\mu(\sigma; \beta / \sqrt{d}, \bX+\bJ_{ij}) - \mu(\sigma; \beta / \sqrt{d}, \bX) - \frac{\partial }{\partial \bX_{ij}}\mu(\sigma; \beta / \sqrt{d}, \bX)\right| \leq \frac{3\beta^2}{d}\cdot\mu(\sigma; \beta / \sqrt{d}, \bX + \xi\bJ_{ij}). \label{eqn:second_der_bound_result}
    \end{equation}
\end{lemma}
\begin{proof}
    Since $\mu$ is infinitely differentiable in each entry of $\bX$, by Taylor's Theorem, there exists $\xi \in [0, 1]$ such that 
    \begin{equation}
    \mu(\sigma; \beta / \sqrt{d}, \bX+\bJ_{ij}) - \mu(\sigma; \beta / \sqrt{d}, \bX) - \frac{\partial }{\partial \bX_{ij}}\mu(\sigma; \beta / \sqrt{d}, \bX) = \frac{1}{2}\cdot \frac{\partial^2 }{\partial \bX_{ij}^2}\mu(\sigma; \beta / \sqrt{d}, \bX + \xi\bJ_{ij}).\label{eqn:second_der_bound_step1}
    \end{equation}
    It remains to bound $ \frac{\partial^2 }{\partial \bX_{ij}^2}\mu(\sigma; \beta / \sqrt{d}, \bX + \xi\bJ_{ij})$. By \eqref{eqn:mu_partial}, we have (omitting $\beta/\sqrt{d}, \bX + \xi\bJ_{ij}$ from the notation for simplicity)
    \begin{align*}
        \frac{\partial^2 }{\partial \bX_{ij}^2}\mu(\sigma) & = \frac{\partial }{\partial \bX_{ij}}\left(\left(-\frac{\beta}{\sqrt{d}}\right) \left(\sigma(i)\sigma(j) - \sum_{\sigma' \in \{-1, 1\}^n}\sigma'(i)\sigma'(j)\mu(\sigma')\right)\mu(\sigma)\right) \\
        & = \left(-\frac{\beta}{\sqrt{d}}\right)^2 \left(-\sum_{\sigma' \in \{-1, 1\}^n}\sigma'(i)\sigma'(j)\left(\sigma'(i)\sigma'(j) - \sum_{\sigma'' \in \{-1, 1\}^n}\sigma''(i)\sigma''(j)\mu(\sigma'')\right)\mu(\sigma')\right)\mu(\sigma) \\
        & \qquad\qquad + \left(-\frac{\beta}{\sqrt{d}}\right)^2 \left(\sigma(i)\sigma(j) - \sum_{\sigma' \in \{-1, 1\}^n}\sigma'(i)\sigma'(j)\mu(\sigma')\right)^2\mu(\sigma)  \\
        & = \left(-\frac{\beta}{\sqrt{d}}\right)^2 \left(-1 + \sum_{\sigma',\sigma'' \in \{-1, 1\}^n}\sigma'(i)\sigma'(j)\sigma''(i)\sigma''(j)\mu(\sigma')\mu(\sigma'')\right)\mu(\sigma) \\
        & \qquad\qquad + \left(-\frac{\beta}{\sqrt{d}}\right)^2 \left(\sigma(i)\sigma(j) - \sum_{\sigma' \in \{-1, 1\}^n}\sigma'(i)\sigma'(j)\mu(\sigma')\right)^2\mu(\sigma). 
    \end{align*}
    Since 
    \begin{equation}
    \left|\sum_{\sigma',\sigma'' \in \{-1, 1\}^n}\sigma'(i)\sigma'(j)\sigma''(i)\sigma''(j)\mu(\sigma')\mu(\sigma'')\right| =  \left|\langle\sigma'(i)\sigma'(j)\sigma''(i)\sigma''(j)\rangle\right| \leq 1
    \end{equation}
    and similarly 
    \begin{equation}
    \left|\sum_{\sigma' \in \{-1, 1\}^n}\sigma'(i)\sigma'(j)\mu(\sigma')\right| = |\langle \sigma'(i)\sigma'(j)\rangle| \leq 1,
    \end{equation}
    it follows that 
    \begin{equation}
    \label{eqn:second_der_bound_step2}
    \left|\frac{\partial^2 }{\partial \bX_{ij}^2}\mu(\sigma)\right| \leq \frac{\beta^2}{d}(2 + 4)\mu(\sigma) = \frac{6\beta^2}{d}\mu(\sigma).
    \end{equation}
    \eqref{eqn:second_der_bound_step1} and \eqref{eqn:second_der_bound_step2} together imply \eqref{eqn:second_der_bound_result}.
\end{proof}

\begin{proposition}\label{prop:xi_estimate}
    Let $\bJ_{ij}$ be the matrix with 1 in the $(i, j)$-th entry and 0 everywhere else. For any $\xi \in \mathbb{R}$, we have
    \begin{equation}
    \e^{-2\beta|\xi|} \cdot \mu(\sigma; \beta, \bX) \leq \mu(\sigma; \beta, \bX + \xi \bJ_{ij}) \leq \e^{2\beta|\xi|}\cdot \mu(\sigma; \beta, \bX)
    \end{equation}
\end{proposition}
\begin{proof}
    We have
    \begin{equation}
    \exp(-\beta H(\sigma; \bX + \xi \bJ_{ij})) = \exp(-\beta H(\sigma; \bX) -\beta H(\sigma; \xi \bJ_{ij})) = \exp(-\beta H(\sigma; \bX))\cdot\e^{-\beta\xi\sigma(i)\sigma(j)},
    \end{equation}
    which implies that   
    \begin{equation}
     \e^{-\beta|\xi|}\exp(-\beta H(\sigma; \bX)) \leq \exp(-\beta H(\sigma; \bX + \xi \bJ_{ij})) \leq \e^{\beta|\xi|}\exp(-\beta H(\sigma; \bX)).
    \end{equation}
    Summing over $\sigma \in \{-1, 1\}^n$, we obtain
    \begin{equation}\label{eq:Z_lipschitz}
     \e^{-\beta|\xi|}Z(\beta, \bX) \leq Z(\beta, \bX + \xi \bJ_{ij}) \leq \e^{\beta|\xi|}Z(\beta, \bX).
    \end{equation}
    Since these quantities are all strictly positive, we have
    \begin{equation}
        \mu(\sigma; \beta, \bX + \xi \bJ_{ij}) = \frac{\exp(-\beta H(\sigma; \bX + \xi \bJ_{ij}))}{Z(\beta, \bX + \xi \bJ_{ij})} \leq \frac{\e^{\beta |\xi|}\exp(-\beta H(\sigma; \bX))}{\e^{-\beta|\xi|}Z(\beta, \bX)} = \e^{2\beta|\xi|}\cdot \mu(\sigma; \beta, \bX)
    \end{equation}
     and similarly
     \begin{equation}
         \e^{-2\beta|\xi|}\cdot \mu(\sigma; \beta, \bX) \leq \mu(\sigma; \beta, \bX + \xi \bJ_{ij}).
     \end{equation}
\end{proof}

\begin{lemma}\label{lem:graph_overlap_energy}
For every $\beta \in \mathbb{R}$, we have
\begin{equation}
\frac{1}{\sqrt{d}n}\E\left[\left\langle H_{d}(\sigma) \right\rangle_{\frac{\beta}{\sqrt{d}}, \bA} \right]= \frac{\sqrt{d}}{2} \cdot \E\left[\left\langle \left(\frac{1}{n}\sum_{i = 1}^n \sigma(i)\right)^2\right\rangle_{\frac{\beta}{\sqrt{d}}, \bA}\right] -\frac{\beta}{2}\E\left[1 - \langle R_{1,2}^2\rangle_{\frac{\beta}{\sqrt{d}}, \bA}\right] + O_d\left(\frac{1}{\sqrt{d}}\right).
\end{equation}
\end{lemma}
\begin{proof}
    For simplicity, we drop the subscripts from $\langle \cdot \rangle$. We have 
    \begin{align}
        \frac{1}{\sqrt{d}n}\E\left[\left\langle H_{d}(\sigma) \right\rangle \right] & = \frac{1}{\sqrt{d}n}\E\left[ \sum_{\sigma \in \{-1, 1\}^n}\sum_{i, j = 1}^n \sigma(i)\sigma(j)\bA_{ij}\cdot \mu\left(\sigma; \beta / \sqrt{d}, \bA\right) \right] \nonumber\\ 
        & = \frac{1}{\sqrt{d}n}\cdot\frac{d}{2n}\cdot \E\left[ \sum_{\sigma \in \{-1, 1\}^n}\sum_{i, j = 1}^n \sigma(i)\sigma(j)\mu\left(\sigma; \beta / \sqrt{d}, \bA + \bJ_{ij}\right) \right] \nonumber\\
        & = \frac{\sqrt{d}}{2n^2}\cdot \E\left[ \sum_{\sigma \in \{-1, 1\}^n}\sum_{i, j = 1}^n \sigma(i)\sigma(j)\mu\left(\sigma; \beta / \sqrt{d}, \bA + \bJ_{ij}\right) \right]. \label{eqn:taylor_step1}
    \end{align}
    Here in the second equality we invoked Lemma~\ref{lemma:stein-chen}. By Lemma~\ref{lemma:taylor_series_for_mu}, for some $\xi_{\sigma, i, j} \in [0, 1]$ we have
    \begin{align}
    &\Bigg|\E\Bigg[ \sum_{\sigma \in \{-1, 1\}^n}\sum_{i, j = 1}^n \sigma(i)\sigma(j)\mu\left(\sigma; \beta / \sqrt{d}, \bA + \bJ_{ij}\right) \Bigg] \nonumber \\
    & \qquad - \E\Bigg[ \sum_{\sigma \in \{-1, 1\}^n}\sum_{i, j = 1}^n \sigma(i)\sigma(j)\left(\mu\left(\sigma; \beta / \sqrt{d}, \bA\right) + \frac{\partial}{\partial \bX_{ij}}\mu\left(\sigma; \beta / \sqrt{d}, \bA\right)\right)\Bigg]\Bigg| \nonumber \\
    \leq\,& \frac{3\beta^2}{d}\cdot \E\left[ \sum_{\sigma \in \{-1, 1\}^n}\sum_{i, j = 1}^n \mu\left(\sigma; \beta / \sqrt{d}, \bA + \xi_{\sigma, i, j}\bJ_{ij}\right) \right] \nonumber \\
    \leq\,& \frac{3\beta^2}{d}\cdot n^2\cdot\e^{2\beta/\sqrt{d}}\cdot\E\Bigg[ \sum_{\sigma \in \{-1, 1\}^n} \mu\left(\sigma; \beta / \sqrt{d}, \bA\right) \Bigg] \nonumber \\
    =\,& \frac{3\beta^2}{d}\cdot n^2\cdot\e^{2\beta/\sqrt{d}}, \label{eqn:taylor_step2}
    \end{align}
    where the last inequality is due to Proposition~\ref{prop:xi_estimate}. We observe that   \begin{equation}\label{eqn:taylor_step3}
        \sum_{\sigma \in \{-1, 1\}^n}\sum_{i, j = 1}^n \sigma(i)\sigma(j)\mu\left(\sigma; \beta / \sqrt{d}, \bA\right)= n^2\cdot \left\langle \left(\frac{1}{n}\sum_{i = 1}^n \sigma(i)\right)^2\right\rangle.
    \end{equation}
    By Proposition~\ref{prop:overlap_partial_mu}, we also have    
    \begin{equation}\label{eqn:taylor_step4}
    \sum_{\sigma \in \{-1, 1\}^n}\sum_{i,j = 1}^n\sigma(i)\sigma(j)\frac{\partial }{\partial \bA_{ij}}\mu\left(\sigma; \beta / \sqrt{d}, \bA\right) = -\frac{\beta}{\sqrt{d}} \cdot n^2 \left(1 - \langle R_{1,2}^2\rangle\right).
    \end{equation}
    The lemma follows by combining \eqref{eqn:taylor_step1}, \eqref{eqn:taylor_step2}, \eqref{eqn:taylor_step3}, and \eqref{eqn:taylor_step4}.
\end{proof}
The next proposition deals with the first term in the above lemma. Note that this term is trivially bounded in the bisection models for all $\beta \in \mathbb{R}$.
\begin{proposition}\label{prop:magnetization_zero}
For every $\beta \geq 0$, we have
\begin{equation}
\lim_{d \to \infty}\sqrt{d} \cdot \limsup_{n \to \infty}\E\left[\left\langle \left(\frac{1}{n}\sum_{i = 1}^n \sigma(i)\right)^2\right\rangle_{\beta/\sqrt{d}, \bA}\right] = 0.
\end{equation}
\end{proposition}

\begin{proof}
Fix an arbitrary $\epsilon > 0$ and let $S(n,\epsilon) = \{\sigma \in \{-1, 1\}^n \mid -\epsilon n \leq \sum_{i = 1}^n \sigma(i) \leq \epsilon n\}$. We have
\begin{align*}
    & \E\left[\left\langle \left(\frac{1}{n}\sum_{i = 1}^n \sigma(i)\right)^2\right\rangle_{\beta /\sqrt{d}, \bA}\right] \\
    =\, & \E\left[\sum_{\sigma \in \{-1, 1\}^n} \left(\frac{1}{n}\sum_{i = 1}^n \sigma(i)\right)^2\cdot \mu\left(\sigma;\beta/\sqrt{d}, \bA\right)\right] \\
    \leq\, &  \E\left[\sum_{\sigma \in S(n,\epsilon / d^{1/4})} \left(\frac{\epsilon}{d^{1/4}}\right)^2 \cdot \mu\left(\sigma;\beta/\sqrt{d}, \bA\right)\right] + \E\left[\sum_{\sigma \in \{-1, 1\}^n 
    \backslash S(n,\epsilon / d^{1/4})}1\cdot \mu\left(\sigma;\beta/\sqrt{d}, \bA\right)\right]\\
    \leq\, &  \left(\frac{\epsilon}{d^{1/4}}\right)^2 + \E\left[\sum_{\sigma \in \{-1, 1\}^n 
    \backslash S(n,\epsilon / d^{1/4})} \mu\left(\sigma;\beta/\sqrt{d}, \bA\right)\right]. \numberthis \label{eq:magnetization_bound}
\end{align*}

Recall that $A_n = \{\sigma \in \{-1, 1\}^n: |\sum_{i = 1}^n\sigma(i)| \leq 1\}$. 
Given $\sigma \in \{-1, 1\}^n \backslash S(n,\epsilon / d^{1/4})$, take a $\tau \in A_n$ such that $R_{1,2}(\tau, \sigma) = \max_{\sigma' \in A_n} R_{1,2}(\sigma', \sigma)$. By 
Proposition~\ref{prop:poisson_concentration} and \eqref{eq:H_tau-sigma}, we have that there exists a constant $C = C(\epsilon) > 0$ such that with probability $1 - 2\cdot\exp(-C d^{1/4}n)$ we have $H_d(\sigma) - H_d(\tau) \geq \frac{\epsilon^2 \sqrt{d}n}{4}$, and consequently 
\begin{equation}
\mu(\sigma; \beta / \sqrt{d}, \bA) \leq \mu(\tau; \beta / \sqrt{d}, \bA) \cdot \exp\left(-\frac{\beta\epsilon^2n}{4}\right).
\end{equation}
By taking expectations, we obtain
\begin{align*}
\E[\mu(\sigma; \beta / \sqrt{d}, \bA)] & \leq \E[\mu(\tau; \beta / \sqrt{d}, \bA)] \cdot \exp\left(-\frac{\beta\epsilon^2n}{4}\right) + 2\cdot\exp(-C d^{1/4}n) \\
& \leq \frac{1}{|A_n|}\cdot \exp\left(-\frac{\beta\epsilon^2n}{4}\right) + 2\cdot\exp(-C d^{1/4}n). \numberthis \label{eq:mu_bound}
\end{align*}
Here in the last inequality we used the fact that $\E[\mu(\tau; \beta / \sqrt{d}, \bA)]$ is the same for all $\tau \in A_n$.

Combining \eqref{eq:magnetization_bound} and \eqref{eq:mu_bound}, we obtain
\begin{equation}
    \sqrt{d}\cdot\E\left[\left\langle \left(\frac{1}{n}\sum_{i = 1}^n \sigma(i)\right)^2\right\rangle_{\beta /\sqrt{d}, \bA}\right] \leq \epsilon^2 + \sqrt{d}\cdot2^n \left(\frac{1}{|A_n|}\cdot\exp\left(-\frac{\beta\epsilon^2n}{4}\right) + 2\cdot\exp(-C d^{1/4}n)\right).
\end{equation}
If $d$ is sufficiently large, the above gives
\begin{equation}
    \sqrt{d} \cdot \limsup_{n \to \infty}\E\left[\left\langle \left(\frac{1}{n}\sum_{i = 1}^n \sigma(i)\right)^2\right\rangle_{\beta /\sqrt{d}, \bA}\right] \leq \epsilon^2.
\end{equation}
The proposition follows since $\epsilon$ is chosen arbitrarily.
\end{proof}

We are now ready to prove Proposition~\ref{lem:overlap_interpolation}.

\begin{proof}[Proof of Proposition~\ref{lem:overlap_interpolation}]
    For part (a), by Lemma~\ref{lemma:SK_overlap_energy}, Lemma~\ref{lem:graph_overlap_energy}, and Proposition~\ref{prop:magnetization_zero}, we have
    \begin{align*}
    &\lim_{d \to \infty}\limsup_{n \to \infty}\frac{1}{n}\left|\E_\bg\left[\langle H_{\SK}(\sigma)\rangle_{\beta,\bg}\right] - \frac{1}{\sqrt{d}}\E_\bA\left[\langle H_{d}(\sigma)\rangle_{\beta/\sqrt{d},\bA}\right]\right| \\
    =\,\,& 
    \lim_{d \to \infty}\left(\frac{\beta}{2}\limsup_{n \to \infty}\left|\E_\bg\left[1 - \langle R_{1,2}^2\rangle_{\beta, \bg}\right]- \E_\bA\left[1 - \langle R_{1,2}^2\rangle_{\frac{\beta}{\sqrt{d}}, \bA}\right]\right| + \frac{\sqrt{d}}{2} \E\left[\left\langle \left(\frac{1}{n}\sum_{i = 1}^n \sigma(i)\right)^2\right\rangle_{\frac{\beta}{\sqrt{d}}, \bA}\right] + O_d\left(\frac{1}{\sqrt{d}}\right)\right)\\
    =\,\,& 
    \frac{\beta}{2} \cdot \lim_{d \to \infty}\limsup_{n \to \infty}\left|\E_\bg\left[\langle R_{1,2}^2\rangle_{\beta, \bg}\right]-\E_\bA\left[\langle R_{1,2}^2\rangle_{\frac{\beta}{\sqrt{d}}, \bA}\right]\right|.
    \end{align*}
    Part (a) then follows by applying Proposition~\ref{lem:average_energy_SK_sparse}. Part (b) follows in a similar manner and we omit the details.
\end{proof}

\section{Disorder chaos for sparse models}
\label{sec:disorder_chaos}

Let $\bg, \bg_t,\bA,\bA_t$ be random matrices as defined in Section~\ref{secOverview} and let $S \subseteq [-1,1]$. We define the following short-hand notation:
\begin{align}
Z^{S}_{\bg,\bg_t} & = \sum_{\substack{\sigma_1, \sigma_2 \in \{-1, 1\}^n\\ R_{1,2}(\sigma_1, \sigma_2) \in S}} \exp\left(-\beta(H(\sigma_1; \bg) + H(\sigma_2; \bg_t))\right),\\
Z^{S,\bis}_{\bg,\bg_t} & = \sum_{\substack{\sigma_1, \sigma_2 \in A_n\\ R_{1,2}(\sigma_1, \sigma_2) \in S}} \exp\left(-\beta(H(\sigma_1; \bg) + H(\sigma_2; \bg_t))\right)
\end{align}
$Z^S_{\bA,\bA_t}$ and $Z^{S,\bis}_{\bA,\bA_t}$ are similarly defined for the sparse models, scaling  $\beta$ by $\sqrt{d}$:
\begin{align}
Z^{S}_{\bA,\bA_t} & = \sum_{\substack{\sigma_1, \sigma_2 \in \{-1, 1\}^n\\ R_{1,2}(\sigma_1, \sigma_2) \in S}} \exp\left(-\frac{\beta}{\sqrt{d}}\cdot(H(\sigma_1; \bA) + H(\sigma_2; \bA_t))\right),\\
Z^{S,\bis}_{\bA,\bA_t} & = \sum_{\substack{\sigma_1, \sigma_2 \in A_n\\ R_{1,2}(\sigma_1, \sigma_2) \in S}} \exp\left(-\frac{\beta}{\sqrt{d}}\cdot(H(\sigma_1; \bA) + H(\sigma_2; \bA_t))\right).
\end{align}

It is known that the SK model exhibits disorder chaos at any temperature.

\begin{theorem}[Theorem 9 in~\cite{chen_variational_2017}]\label{thm:sk_pos_temp_chaos}
    Let $\beta \in \mathbb{R}$, $t > 0$. Fix an arbitrary $\epsilon > 0$. Let $I_\epsilon = [-1, -\epsilon] \cup [\epsilon, 1]$. There exists some constant $K > 0$ such that for every $n \in \mathbb{N}$,
    \begin{equation}    \E\left[\frac{Z^{I_\epsilon}_{\bg, \bg_t}}{Z^{[-1,1]}_{\bg, \bg_t}}\right] \leq K \cdot \exp(- n / K).
    \end{equation}
\end{theorem}

Theorem~\ref{thm:sk_pos_temp_chaos} immediately implies that the overlap of two configurations sampled from the coupled system $\mu_{\bg,\bg_t}$ is nearly zero.

\begin{corollary}
    If $t > 0$, then
    \begin{equation}
    \lim_{n \to \infty} \E\left[\langle R_{1,2}^2\rangle_{\beta, \bg,\bg_t}\right] = 0.
    \end{equation}
\end{corollary}
\begin{proof}
    For any $\epsilon > 0$, we have
    \begin{align*}
        \E\left[\langle R_{1,2}^2\rangle_{\beta, \bg,\bg_t}\right] & \leq \epsilon^2 \cdot \E\left[\frac{Z^{[-
        \epsilon, \epsilon]}_{\bg, \bg_t}}{Z^{[-1,1]}_{\bg, \bg_t}}\right] + 1 \cdot \E\left[\frac{Z^{I_\epsilon}_{\bg, \bg_t}}{Z^{[-1,1]}_{\bg, \bg_t}}\right]  \\ 
        & \leq \epsilon^2 + K \cdot \exp(-n/K).        
    \end{align*}
    Here we used Theorem~\ref{thm:sk_pos_temp_chaos} and the fact that $Z^{[-
        \epsilon, \epsilon]}_{\bg, \bg_t} \leq Z^{[-1,1]}_{\bg, \bg_t}$. It follows that $\limsup_{n \to \infty} \E\left[\langle R_{1,2}^2\rangle_{\beta, \bg,\bg_t}\right] \leq \epsilon^2$ for any arbitrary $\epsilon > 0$, so the corollary follows. 
\end{proof}

The exponentially small fraction in Theorem~\ref{thm:sk_pos_temp_chaos} can be translated into a constant gap between the free energy of coupled and uncoupled systems.
\begin{proposition}\label{prop:coupled_SK_free_energy_gap}
For every $\epsilon > 0$, there exists some constant $K > 0$ such that 
\begin{equation}
\limsup_{n \to \infty} \left(\frac{1}{n}\E\left[\log Z^{I_\epsilon}_{\bg,\bg_t}\right] - \frac{1}{n}\E\left[\log Z^{[-1, 1]}_{\bg,\bg_t}\right] \right)\leq -\frac{1}{K}.
\end{equation}
\end{proposition}
\begin{proof}
    By Theorem~\ref{thm:sk_pos_temp_chaos} and Jensen's inequality
    \begin{align*}
    \frac{1}{n}\E\left[\log Z^{I_\epsilon}_{\bg,\bg_t}\right] - \frac{1}{n}\E\left[\log Z^{[-1, 1]}_{\bg,\bg_t}\right]  =  \frac{1}{n} \E\left[\log \frac{Z^{I_\epsilon}_{\bg,\bg_t}}{Z^{[-1, 1]}_{\bg,\bg_t}}\right] \leq \frac{1}{n} \log \E\left[ \frac{Z^{I_\epsilon}_{\bg,\bg_t}}{Z^{[-1, 1]}_{\bg,\bg_t}}\right] 
    \leq -\frac{1}{K} + O\left(\frac{1}{n}\right).
    \end{align*}
    The proposition follows by taking $n$ to infinity.
\end{proof}

The following theorem, which is a generalized version of Lemma~\ref{lem:free_energy_interpolation}, allows us to transfer this free energy gap to sparse models.

\begin{theorem}[See e.g.~\cite{chen_disorder_2018, chen_suboptimality_2019}]\label{thm:strongerfreeenergycorrespondence}
For any $\beta \in \mathbb{R}$,
    \begin{align}
    \left|\frac{1}{n}\E\left[\log Z^{I_\epsilon,\bis}_{\bg,\bg_t}\right] - \frac{1}{n}\E\left[\log Z^{I_\epsilon,\bis}_{\bA,\bA_t}\right]\right| \leq O_d\left(\frac{1}{\sqrt{d}}\right) + \sqrt{d} \cdot o_n(1).
    \end{align}
\end{theorem}

We are now ready to prove Lemma~\ref{lem:disorder_chaos_sparse}.

\begin{proof}[Proof of Lemma~\ref{lem:disorder_chaos_sparse}]
    As before, we only prove part (a). Fix some arbitrary $\epsilon > 0$.
    By Theorem~\ref{thm:bisection_free_energy}, for both $S = I_\epsilon$ and $S = [-1, 1]$ we have
    \begin{align}
    \lim_{n \to \infty} \frac{1}{n}\left(  \E\left[\log Z^{S,\bis}_{\bg,\bg_t}\right] - \E\left[\log Z^{S}_{\bg,\bg_t}\right]\right) = 0.
    \end{align}
    By Proposition~\ref{prop:coupled_SK_free_energy_gap}, for some constant $K = K(\epsilon)> 0$, we have
\begin{equation}
\limsup_{n \to \infty} \frac{1}{n}\left(\E\left[\log Z^{I_\epsilon, \bis}_{\bg,\bg_t}\right] - \E\left[\log Z^{[-1, 1], \bis}_{\bg,\bg_t}\right] \right)\leq -\frac{1}{K}.
\end{equation}
We can then use Theorem~\ref{thm:strongerfreeenergycorrespondence} to transfer this gap to the sparse model and obtain
\begin{equation}
\lim_{d\to\infty}\limsup_{n \to \infty} \frac{1}{n}\left(\E\left[\log Z^{I_\epsilon, \bis}_{\bA,\bA_t}\right] -\E\left[\log Z^{[-1, 1], \bis}_{\bA,\bA_t}\right]\right) \leq -\frac{1}{K}.
\end{equation}
    By Theorem~\ref{theorem:phi_sparse_limit_bis_coupled}, we then have 
    \begin{equation}
    \lim_{d \to \infty}\limsup_{n \to \infty}\frac{1}{n}\left(\E\left[\log Z^{I_\epsilon}_{\bA,\bA_t}\right] -\E\left[\log Z^{[-1, 1]}_{\bA,\bA_t}\right]\right) \leq -\frac{1}{K}.
    \end{equation}
    This means that if $d$ is sufficiently large, then for all sufficiently large $n$ we have
    \begin{equation}
    \frac{1}{n}\E\left[\log Z^{I_\epsilon}_{\bA,\bA_t}\right] -\frac{1}{n}\E\left[\log Z^{[-1, 1]}_{\bA,\bA_t}\right] \leq -\frac{1}{2K}.
    \end{equation} 
    By Lemma~\ref{lemma:logZ_sparse_concentration_coupled}, we have that with probability at least $1 - o_n(1)$ 
    \begin{equation}
    \frac{1}{n}\log Z^{I_\epsilon}_{\bA,\bA_t} -\frac{1}{n}\log Z^{[-1, 1]}_{\bA,\bA_t} \leq -\frac{1}{4K},
    \end{equation}
    which rearranges to
    \begin{equation}
    \frac{Z^{I_\epsilon}_{\bA,\bA_t}}{Z^{[-1, 1]}_{\bA,\bA_t}} \leq \exp\left( -\frac{n}{4K}\right).
    \end{equation}
    It follows that 
    \begin{equation}    
    \E[\langle R_{1,2}^2\rangle_{\beta/\sqrt{d},\bA,\bA_t}]\leq \epsilon^2 + \E\left[\frac{Z^{I_\epsilon}_{\bA,\bA_t}}{Z^{[-1, 1]}_{\bA,\bA_t}} \right]\leq \epsilon^2 + \exp\left( -\frac{n}{4K}\right) + o_n(1).
    \end{equation}
    By taking $n$ to infinity, we have $\limsup_{n \to \infty}\E[\langle R_{1,2}^2\rangle_{\beta/\sqrt{d},\bA,\bA_t}] \leq \epsilon^2$. This completes the proof.
\end{proof}

\bibliography{references}
\bibliographystyle{alpha}

\appendix
\section{Free energy of bisection models}\label{section:free_energy_bisection}
\subsection{SK model with zero magnetization}

Recall that we defined 
\[
Z^{S}_{\bg,\bg_t} = \sum_{\substack{\sigma_1, \sigma_2 \in \{-1, 1\}^n\\ R_{1,2}(\sigma_1, \sigma_2) \in S}} \exp\left(-\beta(H(\sigma_1; \bg) + H(\sigma_2; \bg_t))\right)\]
and 
\[
Z^{S,\bis}_{\bg,\bg_t} = \sum_{\substack{\sigma_1, \sigma_2 \in A_n\\ R_{1,2}(\sigma_1, \sigma_2) \in S}} \exp\left(-\beta(H(\sigma_1; \bg) + H(\sigma_2; \bg_t))\right).
\]
where $A_n = \{\sigma \in \{-1, 1\}^n: |\sum_{i = 1}^n\sigma(i)| \leq 1\}$ is the set of configurations in which the numbers of $+1$s and $-1$s differ by at most one.
\begin{lemma}\label{lem:z_cover}
    Let $S \subseteq [-1, 1]$ and assume that $n$ is odd. There exists $p(n) = O(n^3)$ independent of $\beta$ such that
    \[
    Z^{S,\bis}_{\bg,\bg_t} \geq \frac{Z^{S}_{\bg,\bg_t}}{p(n)}
    \]
    with probability at least $1/p(n)$.
\end{lemma}
\begin{proof}
    For any two configurations $\sigma, \sigma' \in \{-1, 1\}^n$, let $\sigma \circ \sigma' \in \{-1, 1\}^n$ be the configuration obtained by the entry-wise product of $\sigma, \sigma'$, i.e., $\sigma \circ \sigma' (i) = \sigma(i) \cdot \sigma'(i)$ for every $i \in [n]$. We have the following two observations.
    \begin{itemize}
        \item For any $\sigma, \sigma_1, \sigma_2 \in \{-1, 1\}^n$, $R_{1,2}(\sigma_1, \sigma_2) = R_{1,2}(\sigma_1\circ \sigma, \sigma_2 \circ \sigma)$.
        \item Define $Z^{S,\bis\circ\sigma}_{\bg,\bg_t} = \sum_{\substack{\sigma_1, \sigma_2 \in A_n\\ R_{1,2}(\sigma_1, \sigma_2) \in S}} \exp\left(-\beta(H(\sigma_1 \circ \sigma; \bg) + H(\sigma_2 \circ \sigma; \bg_t))\right)$, then for any $\sigma \in \{-1, 1\}^n$, the distribution of $Z^{S,\bis}_{\bg,\bg_t}$ is the same as the distribution of $Z^{S,\bis\circ\sigma}_{\bg,\bg_t}$.
    \end{itemize}
    We first show that there exists $\tau_1, \ldots, \tau_{p(n)} \in \{-1, 1\}^n$ such that for every $\sigma_1, \sigma_2 \in \{-1, 1\}^n$ with $R_{1,2}(\sigma_1, \sigma_2) \in S$, there exists some $i \in [p(n)]$ for which $\sigma_1 \circ \tau_i, \sigma_2 \circ \tau_i \in A_n$. We show this using the probabilistic method. Let's sample $\tau_1, \ldots, \tau_{p(n)} \in \{-1, 1\}^n$ independently and uniformly at random. For any $\sigma_1, \sigma_2 \in \{-1, 1\}^n$, let $\same(\sigma_1, \sigma_2) = \{i \in [n] \mid \sigma_1(i) = \sigma_2(i)\}$ be the set of indices on which $\sigma_1$ and $\sigma_2$ agree and let $k = |\same(\sigma_1, \sigma_2)|$ be the number of such indices.

    Now, if $n$ is odd, in order for $\sigma_1 \circ \tau$ and $\sigma_2 \circ \tau$ to be both in $A_n$, it is sufficient to have
    \[
    \left|\sum_{i \in \same(\sigma_1, \sigma_2)} (\sigma_1\circ\tau)(i)\right| \leq 1, \text{ and } \left|\sum_{i \in [n] \backslash \same(\sigma_1, \sigma_2)} (\sigma_1\circ\tau)(i)\right| \leq 1.
    \]
    The number of $\tau$'s that satisfy these two conditions is at least
    \[
    \binom{k}{\lfloor k/2 \rfloor}\cdot\binom{n -k}{\lfloor (n-k)/2 \rfloor} \geq \binom{\lfloor n/2 \rfloor}{\lfloor n/4 \rfloor}\cdot\binom{\lceil n/2 \rceil}{\lceil n/4 \rceil} \geq (1 - o(1)) \cdot \frac{4}{\pi n} \cdot 2^n.
    \]
    Here we used the asymptotics $\binom{n}{\lfloor n/2\rfloor} \sim \sqrt{\frac{2}{\pi n}}\cdot 2^n$. It follows that for every $i \in [p(n)]$, 
    \[
    \Pr[\sigma_1 \circ \tau_i, \sigma_2 \circ \tau_i \in A_n] \geq (1 - o(1))\cdot\frac{4}{\pi n}.
    \]
    This implies that 
    \[
    \Pr\left[\exists i \big(\sigma_1 \circ \tau_i, \sigma_2 \circ \tau_i \in A_n\big)\right] \geq 1 - \left(1 - (1 - o(1))\cdot \frac{4}{\pi n}\right)^{p(n)}.
    \]
    We can choose some $p(n) = O(n^3)$ such that the above probability is at least $1 - \exp(-n^2)$ for sufficiently large $n$. Now, via a union bound over all pairs $\sigma_1, \sigma_2$, we have that when $n$ is sufficiently large,
    \[
    \Pr\left[\forall \sigma_1, \sigma_2 \in \{-1, 1\}^n, \exists i \big(\sigma_1 \circ \tau_i, \sigma_2 \circ \tau_i \in A_n\big)\right] \geq 1 - 4^n \cdot \exp\left(-n^2\right) > 0.
    \]
    By the probabilistic method, this means that we have $\forall \sigma_1, \sigma_2 \in \{-1, 1\}^n, \exists i \big(\sigma_1 \circ \tau_i, \sigma_2 \circ \tau_i \in A_n\big)$ for some choice of $\tau_1, \ldots, \tau_{p(n)} \in \{-1, 1\}^n$. Fix such a choice of $\tau_1, \ldots, \tau_{p(n)} \in \{-1, 1\}^n$. We then have
    \begin{equation}\label{eqn:Z_covered1}
    \sum_{i = 1}^{p(n)} Z^{S,\bis\circ\tau_i}_{\bg,\bg_t} \geq Z^{S}_{\bg,\bg_t}.
    \end{equation}
    This is because every term in the partition function on the right hand side also appears in one of the partition functions on the left hand side. Now for the sake of contradiction assume that $\Pr\left[
    Z^{S,\bis}_{\bg,\bg_t} \geq \frac{Z^{S}_{\bg,\bg_t}}{p(n)}\right] < 1/p(n)$, then since $Z^{S,\bis\circ\tau_i}_{\bg,\bg_t}$ has the same distribution as $Z^{S,\bis}_{\bg,\bg_t}$ for every $i$, we can use a union bound to obtain that $\Pr\left[
    \exists i, Z^{S,\bis}_{\bg,\bg_t} \geq \frac{Z^{S}_{\bg,\bg_t}}{p(n)}\right] < 1$. However, this implies that \eqref{eqn:Z_covered1} is violated with positive probability, which gives us a contradiction. The lemma then follows.
\end{proof}

\begin{theorem}\label{thm:bisection_free_energy}
    Let $S \subseteq [-1, 1]$. For any $\beta \in \mathbb{R}$, we have
    \begin{align*}
    \lim_{n \to \infty} \frac{1}{n}  \E\left[\log Z^{S,\bis}_{\bg,\bg_t} - \log Z^{S}_{\bg,\bg_t}\right] = 0.
    \end{align*}
\end{theorem}
\begin{proof}
    Note that increasing $n$ by 1 changes $\E\left[\log Z\right]$ by at most $o(n)$ for both $Z = Z^{S,\bis}_{\bg,\bg_t}$ and $Z = Z^{S}_{\bg,\bg_t}$, so we can assume that $n$ is odd without loss of generality. By Lemma~\ref{lem:z_cover}, we have
    \[
    \log Z^{S,\bis}_{\bg,\bg_t} - \log Z^{S}_{\bg,\bg_t} \geq -C \log n
    \]
    for some constant $C > 0$ with probability at least $\Omega(1/n^3)$. We also know that $\log Z^{S,\bis}_{\bg,\bg_t} \leq \log Z^{S}_{\bg,\bg_t}$, so we have
    \[
    \left|\log Z^{S,\bis}_{\bg,\bg_t} - \log Z^{S}_{\bg,\bg_t}\right| \leq C \log n
    \]
    with probability at least $\Omega(1/n^3)$. By the triangle inequality, we have
    \begin{align}
        &\left|\E\left[\log Z^{S,\bis}_{\bg,\bg_t} - \log Z^{S}_{\bg,\bg_t}\right]\right|  \notag \\
        \leq\,& \left|\E\left[\log Z^{S,\bis}_{\bg,\bg_t}\right] - \log Z^{S,\bis}_{\bg,\bg_t}\right| + \left| \log Z^{S,\bis}_{\bg,\bg_t} - \log Z^{S}_{\bg,\bg_t}\right| + \left|\log Z^{S}_{\bg,\bg_t} -\E\left[\log Z^{S}_{\bg,\bg_t}\right]\right| \label{eqn:Z_triangle_inequality}
    \end{align}
    It is well-known that $Z^{S}_{\bg,\bg_t}$ and $Z^{S,\bis}_{\bg,\bg_t}$ concentrate around their means: there exists $c > 0$ such that for any $t > 0$,   \begin{equation}\label{eqn:Z_concentration}
    \Pr\left[\left|\E\left[\log Z\right] - \log Z\right| > t\right] \leq 2 \cdot \exp\left(-\frac{c\cdot t^2}{n}\right),
    \end{equation}
    where $Z = Z^{S}_{\bg,\bg_t}$ or $Z^{S,\bis}_{\bg,\bg_t}$ (see e.g. Proposition 1.3.5 in~\cite{talagrand_mean_2011}). In particular, we can choose $t = n^{3/4}$ and \eqref{eqn:Z_concentration} becomes  \begin{equation}\label{eqn:Z_concentration2}
    \Pr\left[\left|\E\left[\log Z\right] - \log Z\right| > n^{3/4}\right] \leq 2 \cdot \exp\left(-c \sqrt{n}\right).
    \end{equation}
    It follows that with positive probability, we have $\left|\E\left[\log Z\right] - \log Z\right| \leq n^{3/4}$ for $Z = Z^{S}_{\bg,\bg_t}$ and $Z^{S,\bis}_{\bg,\bg_t}$ and also $\left|\log Z^{S,\bis}_{\bg,\bg_t} - \log Z^{S}_{\bg,\bg_t}\right| \leq C \log n$, which implies that the RHS of \eqref{eqn:Z_triangle_inequality} is $o(n)$ with positive probability. However, since the LHS of \eqref{eqn:Z_triangle_inequality} is a constant, the above reasoning means that it's always $o(n)$. Our theorem follows by dividing \eqref{eqn:Z_triangle_inequality} by $n$.
\end{proof}

\begin{corollary}\label{cor:SK_bis_free_energy}
We have
    \begin{align*}
    \lim_{n \to \infty} \frac{1}{n}  \E\left[\log Z^{\bis}(\beta, \bg) - \log Z(\beta, \bg)\right] = 0.
    \end{align*}    
\end{corollary}
\begin{proof}
    Note that when $t = 1$ and $S = [-1, 1]$, we have $Z^S_{\bg,\bg_t} = Z(\beta, \bg)^2$ and $Z^{S,\bis}_{\bg,\bg_t} = Z^{\bis}(\beta, \bg)^2$. The corollary then follows directly from Theorem~\ref{thm:bisection_free_energy}.
\end{proof}

\subsection{Sparse bisection model}

We need the following folklore concentration result about Poisson random variables (see e.g.~\cite{CanonnePoisson} for a proof). 

\begin{proposition}\label{prop:poisson_concentration}
    Let $X \sim \Po(\lambda)$. For any $t > 0$, we have
    \begin{equation}\label{eq:poisson_concentration_+}
    \Pr[X \geq \lambda + t] \leq \exp\left(-\frac{t^2}{2(\lambda + t)}\right),
    \end{equation}
    and
    \begin{equation}\label{eq:poisson_concentration_-}
    \Pr[X \leq \lambda - t] \leq \exp\left(-\frac{t^2}{2\lambda}\right).
    \end{equation}
\end{proposition}

Given $\sigma \in \{-1, 1\}^n$, $\tau \in A_n$, we say that $\tau$ is a \emph{nearest bisection} for $\sigma$ if $R_{1,2}(\tau, \sigma) = \max_{\tau' \in A_n} R_{1,2}(\tau', \sigma)$. Intuitively, $\tau$ is a \emph{nearest bisection} for $\sigma$ if it can be obtained by flipping the fewest number of bits in $\sigma$.

\begin{lemma}\label{lem:bisection_prob_bound}
    For any $\epsilon > 0$, there exists $\delta > 0$ such that if $d$ is sufficiently large, $\sigma \in \{-1, 1\}^n$, $\tau \in A_n$, and $\tau$ is a nearest bisection for $\sigma$ then
    \begin{equation}\label{eq:bisection_prob_bound}
    \Pr\left[H_d(\tau) - H_d(\sigma) \geq \epsilon \cdot \sqrt{d} \cdot n\right] \leq 2 \cdot \exp\left(-\delta d^{1/4} n\right).
    \end{equation}
\end{lemma}
\begin{proof}
    Let $S = \{i \in [n] : \sigma(i) \neq \tau(i)\}$, $T^+_\tau = \{i \in [n] : \tau(i) = 1\}$ and $T^-_\tau = \{i \in [n] : \tau(i) = -1\}$. For simplicity, let us assume that $n$ is even so $|T^+_\tau| = |T^-_\tau| = n/2$. Let $\Delta = |S| \in [0, n/2]$.
    
    Since $\tau$ is a nearest bisection for $\sigma$, we know that $\tau(i)$ has the same sign for every $i \in S$. Without loss of generality, assume that $\tau(i) = 1$ and $\sigma(i) = -1$ for every $i \in S$ (if not, we can negate both $\sigma$ and $\tau$). We have
    \begin{align*}
        H_d(\tau) - H_d(\sigma) & = \sum_{i, j = 1}^n A_{ij}\left(\tau(i)\tau(j) - \sigma(i)\sigma(j)\right) \\
        & = 2 \sum_{i \in S, j \notin S} A_{ij}\left(\tau(i)\tau(j) - \sigma(i)\sigma(j)\right) \\
        & = 4 \left(\sum_{i \in S, j \in \T^+_\tau \backslash S} A_{ij} - \sum_{i \in S, j \in \T^-_\tau} A_{ij} \right) \numberthis \label{eq:H_tau-sigma}.
    \end{align*}
    Let $U = \sum_{i \in S, j \in \T^+_\tau \backslash S} A_{ij}$ and $V =\sum_{i \in S, j \in \T^-_\tau} A_{ij}$. Since the sum of finitely many independent Poisson variables is still Poisson, we have $U \sim \Po\left(\frac{d\Delta}{4} - \frac{d\Delta^2}{2n}\right)$ and $V \sim \Po\left(\frac{d\Delta}{4}\right)$. If
    \[
    U < \frac{d\Delta}{4} - \frac{d\Delta^2}{2n} + \frac{d\Delta^2}{4n} + \frac{\epsilon \sqrt{d} n}{8} \qquad \text{and} \qquad V > \frac{d\Delta}{4} - \frac{d\Delta^2}{4n} - \frac{\epsilon \sqrt{d} n}{8},
    \]
    then
    \[
    H_d(\tau) - H_d(\sigma) = 4(U - V) < \epsilon \sqrt{d} n,
    \]
    so it follows that 
    \begin{align*}
        \Pr\left[H_d(\tau) - H_d(\sigma) \geq \epsilon \cdot \sqrt{d} \cdot n\right] \leq \Pr\left[U \geq \frac{d\Delta}{4} - \frac{d\Delta^2}{2n} + \frac{d\Delta^2}{4n} + \frac{\epsilon \sqrt{d} n}{8}\right] + \Pr\left[V \leq \frac{d\Delta}{4} - \frac{d\Delta^2}{4n} - \frac{\epsilon \sqrt{d} n}{8}\right].
    \end{align*}
    It remains to bound the two probabilities on the right hand side of the inequality above. By \eqref{eq:poisson_concentration_+},
    \begin{align*}
        \Pr\left[U \geq \frac{d\Delta}{4} - \frac{d\Delta^2}{2n} + \frac{d\Delta^2}{4n} + \frac{\epsilon \sqrt{d} n}{8}\right] & \leq \exp\left(-\frac{\left(\frac{d\Delta^2}{4n} + \frac{\epsilon \sqrt{d} n}{8}\right)^2}{2\left(\frac{d\Delta}{4} - \frac{d\Delta^2}{2n} + \frac{d\Delta^2}{4n} + \frac{\epsilon \sqrt{d} n}{8}\right)}\right) \\
        & \leq \exp\left(-\frac{\left(\frac{d\Delta^2}{4n} + \frac{\epsilon \sqrt{d} n}{8}\right)^2}{2\left(\frac{d\Delta}{4}+\frac{\epsilon \sqrt{d} n}{8}\right)}\right).       
    \end{align*}
    We have three cases.
    \begin{itemize}
        \item $\Delta \leq \frac{\epsilon n}{2\sqrt{d}}$. In this case we have
        \[
        \exp\left(-\frac{\left(\frac{d\Delta^2}{4n} + \frac{\epsilon \sqrt{d} n}{8}\right)^2}{2\left(\frac{d\Delta}{4}+\frac{\epsilon \sqrt{d} n}{8}\right)}\right) \leq \exp\left(-\frac{\left(\frac{\epsilon \sqrt{d} n}{8}\right)^2}{2\left(\frac{\epsilon \sqrt{d} n}{8}+\frac{\epsilon \sqrt{d} n}{8}\right)}\right) = \exp\left(-\frac{\epsilon \sqrt{d}n}{16}\right).
        \]
        \item $\frac{\epsilon n}{2\sqrt{d}} \leq \Delta \leq \sqrt{2\epsilon} \cdot d^{-1/4} n$. In this case we have
        \[
        \exp\left(-\frac{\left(\frac{d\Delta^2}{4n} + \frac{\epsilon \sqrt{d} n}{8}\right)^2}{2\left(\frac{d\Delta}{4}+\frac{\epsilon \sqrt{d} n}{8}\right)}\right) \leq \exp\left(-\frac{\left(\frac{\epsilon \sqrt{d} n}{8}\right)^2}{2\left(\frac{d\Delta}{4}+\frac{d\Delta}{4}\right)}\right) = \exp\left(-\frac{\epsilon^2 n^2}{64\Delta}\right) \leq \exp\left(-\frac{\epsilon^2}{64\sqrt{2\epsilon}}\cdot d^{1/4}n\right).
        \]
        \item $\Delta \geq \sqrt{2\epsilon} \cdot d^{-1/4} n$. In this case we have
        \[
        \exp\left(-\frac{\left(\frac{d\Delta^2}{4n} + \frac{\epsilon \sqrt{d} n}{8}\right)^2}{2\left(\frac{d\Delta}{4}+\frac{\epsilon \sqrt{d} n}{8}\right)}\right) \leq \exp\left(-\frac{\left(\frac{d\Delta^2}{4n}\right)^2}{2\left(\frac{d\Delta}{4}+\frac{d\Delta}{4}\right)}\right) = \exp\left(-\frac{d\Delta^3}{16n^2}\right) \leq  \exp\left(-\frac{(2\epsilon)^{3/2}}{16}\cdot d^{1/4} n\right).
        \]
    \end{itemize}
    It follows that there exists $\delta_1 = \delta_1(\epsilon) > 0$ such that 
    \[
    \Pr\left[U \geq \frac{d\Delta}{4} - \frac{d\Delta^2}{2n} + \frac{d\Delta^2}{4n} + \frac{\epsilon \sqrt{d} n}{8}\right] \leq \exp\left(-\frac{\left(\frac{d\Delta^2}{4n} + \frac{\epsilon \sqrt{d} n}{8}\right)^2}{2\left(\frac{d\Delta}{4}+\frac{\epsilon \sqrt{d} n}{8}\right)}\right) \leq \exp\left(-\delta_1 \cdot d^{1/4} n\right).
    \]
    Using a similar analysis (with \eqref{eq:poisson_concentration_-} this time) we can show that
    \[
    \Pr\left[V \leq \frac{d\Delta}{4} - \frac{d\Delta^2}{4n} - \frac{\epsilon \sqrt{d} n}{8}\right] \leq \exp\left(-\delta_2 \cdot d^{1/4} n\right)
    \]
    for some $\delta_2 = \delta_2(\epsilon) > 0$. This completes the proof.
\end{proof}

\begin{lemma}\label{lemma:sigma_to_tau_pairs}
    Assume that $n$ is odd. There exists a function $f: \{-1, 1\}^n \times \{-1, 1\}^n \to A_n \times A_n$ satisfying the following properties:
    \begin{itemize}
        \item For every $\sigma_1, \sigma_2 \in \{-1, 1\}^n$, if $f(\sigma_1, \sigma_2) = (\tau_1, \tau_2)$, then $\tau_i$ is a nearest bisection for $\sigma_i$ for $i = 1, 2$, and furthermore $R_{1,2}(\sigma_1, \sigma_2) = R_{1,2}(\tau_1, \tau_2)$.
        \item There exists a constant $C > 0$ independent of $n$ such that for every $\tau_1, \tau_2 \in A_n$, $|f^{-1}(\tau_1, \tau_2)| \leq C n^{3/2}$. 
    \end{itemize}
\end{lemma}
\begin{proof}
    For every $\sigma_1, \sigma_2 \in \{-1, 1\}^n$, let $N(\sigma_1, \sigma_2) \subseteq A_n \times A_n$ be the set of pairs $(\tau_1, \tau_2)$ where $\tau_i$ is a nearest bisection for $\sigma_i$ for $i = 1, 2$ and $R_{1,2}(\sigma_1, \sigma_2) = R_{1,2}(\tau_1, \tau_2)$. It can be verified that when $n$ is odd, the set $N(\sigma_1, \sigma_2)$ is always non-empty. Consider a random function $f$ where for every pair $\sigma_1, \sigma_2 \in \{-1, 1\}^n$, we independently pick an element from $N(\sigma_1, \sigma_2)$ uniformly at random as $f(\sigma_1, \sigma_2)$. 
    
    For any pair $(\tau_1, \tau_2)$, let $X_{\tau_1, \tau_2, \sigma_1, \sigma_2}$ be the indicator random variable for the event $f(\sigma_1, \sigma_2) = (\tau_1, \tau_2)$ and let $X_{\tau_1, \tau_2} = \sum_{\sigma_1, \sigma_2 \in \{-1, 1\}^n} X_{\tau_1, \tau_2, \sigma_1, \sigma_2}$. We have $X_{\tau_1, \tau_2} = |f^{-1}(\tau_1, \tau_2)|$, and for any value $r \in [-1, 1]$, we have
    \begin{equation}
        \sum_{\substack{\tau_1, \tau_2 \in A_n \\ R_{1,2}(\tau_1, \tau_2) = r}} X_{\tau_1, \tau_2} = \big|\left\{(\sigma_1, \sigma_2) \in \{-1, 1\}^n \times \{-1, 1\}^n \mid R_{1,2}(\sigma_1, \sigma_2) = r\right\}\big|
    \end{equation} 
    
    It follows from symmetry that for any $\tau_1, \tau_2 \in A_n$ with $R_{1,2}(\tau_1, \tau_2) = r$ we have
    \begin{equation}
    \E[X_{\tau_1, \tau_2}] = \frac{\big|\left\{(\sigma_1, \sigma_2) \in \{-1, 1\}^n \times \{-1, 1\}^n \mid R_{1,2}(\sigma_1, \sigma_2) = r\right\}\big|}{\big|\left\{(\tau_1, \tau_2) \in A_n \times A_n \mid R_{1,2}(\tau_1, \tau_2) = r\right\}\big|} = \frac{2^n\cdot \binom{n}{\frac{(1-r)n}{2}}}{|A_n|\cdot \binom{\left\lfloor \frac{n}{2} \right\rfloor}{\left\lfloor\frac{(1-r)n}{4}\right\rfloor}\binom{\left\lceil \frac{n}{2} \right\rceil}{\left\lceil\frac{(1-r)n}{4}\right\rceil}}.
    \end{equation}
     By elementary estimates, there exist constants $c_1, c_2 > 0$ such that $c_1 \cdot \sqrt{n} \leq \E[X_{\tau_1, \tau_2}] \leq c_2 \cdot n$, for every $\tau_1, \tau_2 \in A_n$.
Since $X_{\tau_1, \tau_2}$ is a sum of independent $\{0, 1\}$-valued random variables, we can apply a Chernoff bound and obtain that for every $\delta > 0$
\[
\Pr\left[X_{\tau_1, \tau_2} \geq (1 + \delta)\E[X_\tau]\right] \leq \exp\left(-\frac{\delta^2\E[X_\tau]}{\delta + 2}\right) \leq \exp\left(-\frac{\delta^2\cdot c_1\sqrt{n}}{\delta + 2}\right).
\]
By a union bound,
\[
\Pr\left[\exists \tau_1,\tau_2 \in A_n, X_{\tau_1, \tau_2} \geq (1 + \delta)\E[X_{\tau_1, \tau_2}]\right] \leq 4^n\cdot\exp\left(-\frac{\delta^2\cdot c_1\sqrt{n}}{\delta + 2}\right).
\]
We can then pick some sufficiently large constant $c > 0$ and set $\delta = c\sqrt{n}$ such that
\[
4^n\cdot\exp\left(-\frac{\delta^2\cdot c_1\sqrt{n}}{\delta + 2}\right) \leq 4^n\cdot\exp\left(-\frac{\delta^2\cdot c_1\sqrt{n}}{2\delta}\right) = 4^n\cdot\exp\left(-\frac{c \cdot c_1}{2}\cdot n\right) < 1.
\]
It follows that there exists some choice of $f$ such that for all $\tau_1,\tau_2 \in A_n$
\[
X_{\tau_1,\tau_2} < (1 + \delta)\E[X_{\tau_1,\tau_2}] \leq (1 + c\sqrt{n}) \cdot c_2 n \leq Cn^{3/2},
\]
where we pick $C = 2\cdot c \cdot c_2$.
\end{proof}

Recall that earlier we defined the following shorthand notation for coupled systems:
\begin{align}
Z^{S}_{\bA,\bA_t} & = \sum_{\substack{\sigma_1, \sigma_2 \in \{-1, 1\}^n\\ R_{1,2}(\sigma_1, \sigma_2) \in S}} \exp\left(-\frac{\beta}{\sqrt{d}}\cdot(H(\sigma_1; \bA) + H(\sigma_2; \bA_t))\right),\\
Z^{S,\bis}_{\bA,\bA_t} & = \sum_{\substack{\sigma_1, \sigma_2 \in A_n\\ R_{1,2}(\sigma_1, \sigma_2) \in S}} \exp\left(-\frac{\beta}{\sqrt{d}}\cdot(H(\sigma_1; \bA) + H(\sigma_2; \bA_t))\right).
\end{align}
Here $\bA = \bA^{(1-t)} + \bA^{(t, 1)}$ and $\bA_t = \bA^{(1-t)} + \bA^{(t, 2)}$, where $\bA^{(1-t)}, \bA^{(t, 1)}, \bA^{(t, 2)}$ are three independent random matrices such that $\bA^{(1-t)}_{ij} \sim \Po((1-t)d/(2n))$, $\bA^{(t, 1)}_{ij}, \bA^{(t, 2)}_{ij} \sim \Po(td/(2n))$ independently for all $i,j$. In the following proofs, we will sometimes also use $Z^{S}(\bA,\bA_t)$ for $Z^{S}_{\bA,\bA_t}$ and $Z^{S,\bis}(\bA,\bA_t)$ for $Z^{S,\bis}_{\bA,\bA_t}$.

\begin{lemma}\label{lemma:logZ_sparse_concentration_coupled}
    Fix $\beta, d > 0$ and $t \in [0, 1]$. For every $\delta \in (0, 1/2)$, there exists a constant $C = C(\beta, d)$ such that for all sufficiently large $n$
    \begin{equation}\label{eq:logZ_sparse_concentration_coupled}
    \Pr\left[\left|\log Z^{S}_{\bA,\bA_t} - \E[\log Z^{S}_{\bA,\bA_t}]\right| \geq n^{1/2 + \delta}\right] \leq \exp\left(-C\cdot n^{2\delta}\right)
    \end{equation}
    and
    \begin{equation}\label{eq:logZ_bis_sparse_concentration_coupled}
    \Pr\left[\left|\log Z^{S,\bis}_{\bA,\bA_t} - \E[\log Z^{S,\bis}_{\bA,\bA_t}]\right| \geq n^{1/2 + \delta}\right] \leq \exp\left(-C\cdot n^{2\delta}\right).
    \end{equation}
\end{lemma}
\begin{proof}
    We only prove \eqref{eq:logZ_sparse_concentration_coupled} since the proof for \eqref{eq:logZ_bis_sparse_concentration_coupled} is essentially the same. Let $\left(\bX^{(1-t)}_\ell\right)_{\ell \in \mathbb{N}}$, $\left(\bX^{(t, 1)}_\ell\right)_{\ell \in \mathbb{N}}$, and $\left(\bX^{(t, 2)}_\ell\right)_{\ell \in \mathbb{N}}$ be three independent sequences of i.i.d random $n\times n$ matrices where each of these matrices has exactly one entry being 1 which is chosen uniformly from the $n^2$ entries and all other entries are 0. Let 
    \[
    \bB^{(1-t)}_\ell = \sum_{i = 1}^\ell \bX^{(1-t)}_i, \quad\bB^{(t,1)}_\ell = \sum_{i = 1}^\ell \bX^{(t,1)}_i, \quad\bB^{(t,2)}_\ell = \sum_{i = 1}^\ell \bX^{(t,2)}_i.
    \]
    For notational convenience, let us write $Z^{S}_{\ell, p, q}:=Z^{S}\left(\bB^{(1-t)}_{\ell} + \bB^{(t,1)}_{p}, \bB^{(1-t)}_{\ell} + \bB^{(t,2)}_{q}\right)$. We have for every $\ell_1, \ell_2, p_1, p_2, q_1, q_2 \in \mathbb{N}$,
    \begin{equation}
    \left|\log Z^{S}_{\ell_1, p_1, q_1} - \log Z^{S}_{\ell_2, p_2, q_2}\right| \leq \frac{\beta}{\sqrt{d}} \cdot \left(2|\ell_1 - \ell_2| + |p_1 - p_2| + |q_1 - q_2|\right).
    \end{equation}
    This immediately implies that
    \begin{equation}\label{eq:bB_l_bB_k}
    \left|\E\left[\log Z^{S}_{\ell_1, p_1, q_1}\right] - \E\left[\log Z^{S}_{\ell_2, p_2, q_2}\right]\right| \leq \frac{\beta}{\sqrt{d}} \cdot \left(2|\ell_1 - \ell_2| + |p_1 - p_2| + |q_1 - q_2|\right).
    \end{equation}
    Since the sum of independent Poisson variables is still Poisson, we have that 
    \[
    \sum_{i, j} \bA^{(1-t)}_{i, j} \sim \Po\left((1-t)\cdot \frac{dn}{2}\right), \quad \sum_{i, j} \bA^{(t,1)}_{i, j} \sim \Po\left(t\cdot \frac{dn}{2}\right), \quad\sum_{i, j} \bA^{(t,2)}_{i, j} \sim \Po\left(t\cdot \frac{dn}{2}\right).
    \]
    By properties of the Poisson distribution, the random variable $\log Z^{S}_{\bA, \bA_t}$ conditioned on the event \begin{center}``$\sum_{i, j} \bA^{(1-t)}_{i,j} = \ell$ and $\sum_{i, j} \bA^{(t,1)}_{i,j} = p$ and $\sum_{i, j} \bA^{(t,2)}_{i,j} = q$''\end{center} for any $\ell, p, q \in \mathbb{N}$ has the same distribution as $\log Z^{S}_{\ell, p, q}$. Let $P_{\ell, p, q}$ be the probability of this event. Let $m_{1-t} = \lfloor (1-t)dn/2 \rfloor$ and $m_{t} = \lfloor tdn/2 \rfloor$. We have
    \begin{align*}
    & \left|\E\Big[\log Z^{S}_{\bA, \bA_t}\Big] - \E\Big[\log Z^{S}_{m_{1-t}, m_t, m_t}\Big]\right| \\
    =\, &\Bigg|\Bigg(\sum_{\ell, p, q \geq 0}\E\Big[\log Z^{S}_{\ell, p, q}\Big]\cdot P_{\ell, p, q}\Bigg) - \E\Big[\log Z^{S}_{m_{1-t}, m_t, m_t}\Big]\Bigg| \\
    \leq \, & \sum_{\ell,p,q \geq 0} P_{\ell, p, q} \cdot \Bigg|\E\Big[\log  Z^{S}_{\ell, p, q}\Big] - \E\Big[\log Z^{S}_{m_{1-t}, m_t, m_t}\Big]\Bigg| \\
    \leq \, & \sum_{\ell,p,q \geq 0}  P_{\ell, p, q} \cdot \frac{\beta}{\sqrt{d}} \cdot \left(2|\ell - m_{1-t}| + |p - m_t| + |q - m_t|\right) \\
    = \, & \frac{\beta}{\sqrt{d}} \cdot \E\Bigg[2\Bigg|\sum_{i, j}\bA_{i, j}^{(1-t)} - m_{1-t}\Bigg| + \Bigg|\sum_{i, j}\bA_{i, j}^{(t,1)} - m_{t}\Bigg| + \Bigg|\sum_{i, j}\bA_{i, j}^{(t,2)} - m_{t}\Bigg|\Bigg] \\
    \leq \, & \frac{\beta}{\sqrt{d}} \cdot \left(2\sqrt{\E\Bigg[\Bigg(\sum_{i, j}\bA_{i, j}^{(1-t)} - m_{1-t}\Bigg)^2\Bigg]} + \sqrt{\E\Bigg[\Bigg(\sum_{i, j}\bA_{i, j}^{(t,1)} - m_{t}\Bigg)^2\Bigg]} + \sqrt{\E\Bigg[\Bigg(\sum_{i, j}\bA_{i, j}^{(t,2)} - m_{t}\Bigg)^2\Bigg]} \right)\\
    \leq \, & 2\frac{\beta}{\sqrt{d}}\cdot \left(\sqrt{\frac{(1-t)dn}{2} + 1} + \sqrt{\frac{tdn}{2} + 1}\right), \numberthis \label{eq:bA_bB_m}
    \end{align*}
     where in the last inequality we used the fact that $\Po(\lambda)$ has variance $\lambda$. Combining \eqref{eq:bB_l_bB_k} and \eqref{eq:bA_bB_m} we obtain
    \begin{align*}
        &\Big| \E\big[\log Z^{S}_{\bA, \bA_t}\big] - \E\big[\log Z^{S}_{\ell, p, q}\big]\Big| \\
        \leq\, & \frac{\beta}{\sqrt{d}} \left(2\sqrt{\frac{(1-t)dn}{2} + 1} + 2\sqrt{\frac{tdn}{2} + 1} + 2\left|\ell - \left\lfloor \frac{(1-t)dn}{2}\right\rfloor\right| + \left|p - \left\lfloor \frac{tdn}{2}\right\rfloor\right| +  \left|q - \left\lfloor \frac{tdn}{2}\right\rfloor\right|\right).
    \end{align*}
    In particular, if $\left|\ell - \left\lfloor \frac{(1-t)dn}{2}\right\rfloor\right| \leq \frac{\sqrt{d}n^{1/2 + \delta}}{18\beta}$,  $\left|p - \left\lfloor \frac{tdn}{2}\right\rfloor\right| \leq \frac{\sqrt{d}n^{1/2 + \delta}}{9\beta}$ and  $\left|q - \left\lfloor \frac{tdn}{2}\right\rfloor\right| \leq \frac{\sqrt{d}n^{1/2 + \delta}}{9\beta}$, then if $n$ is sufficiently large
    \begin{equation}
        \Big| \E\big[\log Z^{S}_{\bA, \bA_t}\big] - \E\big[\log Z^{S}_{\ell, p, q}\big]\Big| \leq 4\frac{\beta}{\sqrt{d}} \sqrt{\frac{dn}{2} + 1} + \frac{1}{3}\cdot n^{1/2+\delta} < \frac{1}{2} \cdot n^{1/2 + \delta}.
    \end{equation}
     By Azuma's inequality, we have that for every $\ell, p, q \in \mathbb{N}$
    \begin{equation}
    \Pr\left[\left|\log Z^{S}_{\ell, p, q} - \E\Big[\log Z^{S}_{\ell, p, q}\Big]\right| \geq t \right] \leq \exp\left(- \frac{dt^2}{2(2\ell + p + q)  \beta^2}\right).
    \end{equation}
    It follows that for any $\ell, p, q$ with $\left|\ell - \left\lfloor \frac{(1-t)dn}{2}\right\rfloor\right| \leq \frac{n^{1/2 + \delta}}{18\beta}$,  $\left|p - \left\lfloor \frac{tdn}{2}\right\rfloor\right| \leq \frac{n^{1/2 + \delta}}{9\beta}$ and  $\left|q - \left\lfloor \frac{tdn}{2}\right\rfloor\right| \leq \frac{n^{1/2 + \delta}}{9\beta}$, if $n$ is sufficiently large we have
    \begin{align*}
        & \Pr\left[ \left|\log  Z^{S}_{\ell, p, q} - \E\big[\log Z^{S}_{\bA, \bA_t}\big]\right| \geq n^{1/2 + \delta}\right] \\
        \leq\, & \Pr\left[ \left|\log  Z^{S}_{\ell, p, q} - \E\big[\log Z^{S}_{\ell, p, q}\big]\right| \geq n^{1/2 + \delta} - \Big| \E\big[\log Z^{S}_{\bA, \bA_t}\big] - \E\big[\log Z^{S}_{\ell, p, q}\big]\Big|\right] \\
        \leq\, & \Pr\left[ \left|\log Z^{S}_{\ell, p, q} - \E\big[\log Z^{S}_{\ell, p, q}\big]\right| \geq \frac{1}{2} \cdot n^{1/2 + \delta} \right] \\
        \leq \, & \exp\left(- \frac{dn^{1 + 2\delta}}{8\left(\frac{n^{1/2 + \delta}}{3\beta} + \left\lfloor dn\right\rfloor\right)\beta^2}\right)\\
        \leq \, & \exp\left(- C_1 \cdot n^{2\delta}\right) 
    \end{align*}
    where $C_1 = C_1(\beta , d) > 0$. By Proposition~\ref{prop:poisson_concentration}, we can also find $C_2 = C_2(\beta, d) > 0$ such that for all sufficiently large $n$
    \begin{align}
        \Pr\left[\left|\sum_{i, j = 1}^n \bA^{(1-t)}_{i, j} - \lfloor (1-t)dn/2\rfloor\right| \geq \frac{\sqrt{d}n^{1/2 + \delta}}{18\beta}\right] \leq \exp(-C_2 \cdot n^{2\delta}), \\
        \Pr\left[\left|\sum_{i, j = 1}^n \bA^{(t,1)}_{i, j} - \lfloor tdn/2\rfloor\right| \geq \frac{\sqrt{d}n^{1/2 + \delta}}{9\beta}\right] \leq \exp(-C_2 \cdot n^{2\delta}), \\
        \Pr\left[\left|\sum_{i, j = 1}^n \bA^{(t,2)}_{i, j} - \lfloor tdn/2\rfloor\right| \geq \frac{\sqrt{d}n^{1/2 + \delta}}{9\beta}\right] \leq \exp(-C_2 \cdot n^{2\delta}).
    \end{align}
    We can therefore conclude that there exists $C = C(\beta, d) > 0$ such that if $n$ is suffciently large, 
    \begin{align*}        
    & \Pr\left[\left|\log Z^{S}_{\bA, \bA_t} - \E[\log Z^{S}_{\bA, \bA_t}]\right| \geq n^{1/2 + \delta}\right] \\
    \leq \, & \Pr\left[\left|\sum_{i, j = 1}^n \bA^{(1-t)}_{i, j} - \lfloor (1-t)dn/2\rfloor\right| \geq \frac{\sqrt{d}n^{1/2 + \delta}}{18\beta}\right] + \Pr\left[\left|\sum_{i, j = 1}^n \bA^{(t,1)}_{i, j} - \lfloor tdn/2\rfloor\right| \geq \frac{\sqrt{d}n^{1/2 + \delta}}{9\beta}\right] \\
    & \qquad + \Pr\left[\left|\sum_{i, j = 1}^n \bA^{(t,2)}_{i, j} - \lfloor tdn/2\rfloor\right| \geq \frac{\sqrt{d}n^{1/2 + \delta}}{9\beta}\right] \\
    & \qquad + \sum_{\substack{\ell: \left|\ell - \left\lfloor \frac{(1-t)dn}{2}\right\rfloor\right| \leq \frac{\sqrt{d}n^{1/2 + \delta}}{18\beta} \\ p: \left|p - \left\lfloor \frac{tdn}{2}\right\rfloor\right| \leq \frac{\sqrt{d}n^{1/2 + \delta}}{9\beta}\\ q: \left|p - \left\lfloor \frac{tdn}{2}\right\rfloor\right| \leq \frac{\sqrt{d}n^{1/2 + \delta}}{9\beta}} } P_{\ell, p, q} \cdot\Pr\left[ \left|\log Z^{S}_{\ell, p, q}  - \E\big[\log Z^{S}_{\bA, \bA_t} \big]\right| \geq n^{1/2 + \delta}\right] \\
    \leq \, & 3\exp(-C_2 \cdot n^{2\delta}) +  \sum_{\substack{\ell: \left|\ell - \left\lfloor \frac{(1-t)dn}{2}\right\rfloor\right| \leq \frac{\sqrt{d}n^{1/2 + \delta}}{18\beta} \\ p: \left|p - \left\lfloor \frac{tdn}{2}\right\rfloor\right| \leq \frac{\sqrt{d}n^{1/2 + \delta}}{9\beta}\\ q: \left|p - \left\lfloor \frac{tdn}{2}\right\rfloor\right| \leq \frac{\sqrt{d}n^{1/2 + \delta}}{9\beta}} } P_{\ell, p, q}\cdot\exp\left(- C_1 \cdot n^{2\delta}\right) \\
    \leq \, & 3\exp(-C_2 \cdot n^{2\delta}) + \exp\left(- C_1 \cdot n^{2\delta}\right) \\
    \leq \, & \exp(-C \cdot n^{2\delta}). \qedhere
    \end{align*}
\end{proof}

\begin{theorem}\label{theorem:phi_sparse_limit_bis_coupled}
    Let $0 \leq a < b$. For every $\epsilon > 0$, if $d$ is sufficiently large, then for every $\beta \in [a, b]$, 
    \begin{equation}
    \limsup_{n \to \infty}  \frac{1}{n}\left| \E\log Z^{S}_{\bA,\bA_t} - \E\log Z^{S,\bis}_{\bA,\bA_t}\right| \leq \epsilon.
    \end{equation}    
\end{theorem}

\begin{proof}
    Fix $\epsilon > 0$. Let $f$ be the function specified in Lemma~\ref{lemma:sigma_to_tau_pairs}. By Lemma~\ref{lem:bisection_prob_bound}, there exists $\delta>0$ such that for all $\sigma_1,\sigma_2 \in \{-1, 1\}^n$, if $f(\sigma_1, \sigma_2) = (\tau_1,\tau_2)$ and $d$ is sufficiently large then
    \begin{equation}\label{eq:tau_1_good}
    \Pr\left[H(\tau_1; \bA) - H(\sigma_1;\bA) \geq \epsilon \cdot \sqrt{d} \cdot n\right] \leq 2 \cdot \exp\left(-\delta d^{1/4} n\right)
    \end{equation}
    and
    \begin{equation}\label{eq:tau_2_good}
    \Pr\left[H(\tau_2; \bA_t) - H(\sigma_2;\bA_t) \geq \epsilon \cdot \sqrt{d} \cdot n\right] \leq 2 \cdot \exp\left(-\delta d^{1/4} n\right).
    \end{equation}
    By the union bound, the probability that there exists such a pair $\sigma_1,\sigma_2 \in \{-1, 1\}^n$ for which the events in \eqref{eq:tau_1_good} or \eqref{eq:tau_2_good} happens is at most $4^{n+1} \cdot \exp\left(-\delta d^{1/4} n\right)$. By making $d$ sufficiently large we can make this probability vanishingly small. If no such pair exists, then we have
    \begin{align*}
        Z^{S}_{\bA,\bA_t} & = \sum_{\sigma_1,\sigma_2\in \{-1, 1\}^n:R_{1,2}(\sigma_1,\sigma_2)\in S} \exp\left(-\frac{\beta}{\sqrt{d}} \cdot \left(H(\sigma_1; \bA) + H(\sigma_2; \bA_t)\right)\right) \\ 
        & \leq \sum_{\substack{\sigma_1,\sigma_2\in \{-1, 1\}^n:R_{1,2}(\sigma_1,\sigma_2)\in S \\ f(\sigma_1, \sigma_2) = (\tau_1,\tau_2)}} \exp\left(-\frac{\beta}{\sqrt{d}} \cdot \left( H(\tau_1; \bA) + H(\tau_2; \bA_t) -  2\epsilon \sqrt{d} n\right)\right) \\
        & = \sum_{\substack{\sigma_1,\sigma_2\in \{-1, 1\}^n:R_{1,2}(\sigma_1,\sigma_2)\in S \\ f(\sigma_1, \sigma_2) = (\tau_1,\tau_2)}} \exp\left(-\frac{\beta}{\sqrt{d}} \cdot \left(H(\tau_1; \bA) + H(\tau_2; \bA_t)\right) \right) \cdot \exp\left(2\epsilon \beta n\right) \\
        & \leq Cn^{3/2} \cdot \sum_{\tau_1,\tau_2\in \{-1, 1\}^n:R_{1,2}(\tau_1,\tau_2)\in S} \exp\left(-\frac{\beta}{\sqrt{d}} \cdot \left(H(\tau_1; \bA) + H(\tau_2; \bA_t)\right)  \right) \cdot \exp\left(2\epsilon \beta n\right) \\
        & = Cn^{3/2}  \cdot Z^{S,\bis}_{\bA,\bA_t} \cdot \exp\left(2\epsilon \beta n\right).
    \end{align*}
    Here $C > 0$ is the constant from Lemma~\ref{lemma:sigma_to_tau_pairs} for which we have $|f^{-1}(\tau_1, \tau_2)| \leq C n^{3/2}$ for every $\tau_1,\tau_2 \in A_n$. It follows that
    \[
    \log Z^{S,\bis}_{\bA,\bA_t} \leq \log Z^{S}_{\bA,\bA_t} \leq \log Z^{S,\bis}_{\bA,\bA_t} + 2\epsilon \beta n + \frac{3}{2} \log (Cn).
    \]
    By Lemma~\ref{lemma:logZ_sparse_concentration_coupled}, if $d$ is sufficiently large, then we have with positive probability all three of the following events happen:
    \begin{itemize}
        \item $\left|\log Z^{S}_{\bA,\bA_t} - \E[\log Z^{S}_{\bA,\bA_t}]\right| < n^{3/4}$,
        \item $\left|\log Z^{S, \bis}_{\bA,\bA_t} - \E[\log Z^{S, \bis}_{\bA,\bA_t}]\right| < n^{3/4}$,
        \item $\left|\log Z^{S}_{\bA,\bA_t} - \log Z^{S,\bis}_{\bA,\bA_t}\right| < 2\epsilon \beta n + \frac{3}{2}\log (Cn)$.
    \end{itemize}
    By the triangle inequality, when these three events hold we also have    
    \begin{equation}\label{eq:ElogZ_bound}
        \left|\E\left[\log Z^{S,\bis}_{\bA,\bA_t} - \log Z^{S}_{\bA,\bA_t}\right]\right| \leq 2n^{3/4} + 2\epsilon \beta n + \log (Cn).
    \end{equation}
    Since the both sides of \eqref{eq:ElogZ_bound} are constant, we have that if it holds with positive probability then it must hold always. We then have that if $d$ is sufficiently large
    \begin{equation}
        \limsup_{n \to \infty} \left|\frac{1}{n}  \E\left[\log Z^{S,\bis}_{\bA,\bA_t} - \log Z^{S}_{\bA,\bA_t}\right]\right| \leq 2\epsilon \beta \leq 2\epsilon b.
    \end{equation}
    This completes the proof since $\epsilon$ is arbitrarily chosen.
\end{proof}

By letting $t = 0$ we immediately obtain the following.

\begin{theorem}\label{theorem:phi_sparse_limit_bis}
    Let $0 \leq a < b$. For every $\epsilon > 0$,  if $d$ is sufficiently large, then for every $\beta \in [a, b]$, 
    \begin{equation}\label{eq:phi_sparse_limsup}
    \limsup_{n \to \infty} \left|\Phi_{n,d} \left(\frac{\beta}{\sqrt{d}}\right)  - \Phi_{n,d}^\bis \left(\frac{\beta}{\sqrt{d}}\right) \right| \leq \epsilon.
    \end{equation}    
\end{theorem}
\end{document}